\documentclass[11pt]{amsart}

\usepackage{amsmath}
\usepackage{amsfonts}
\usepackage{amssymb}
\usepackage{eucal}
\usepackage[all]{xy}
 \usepackage[pagebackref]{hyperref}

\def\dfrac{\displaystyle \frac}

\newtheorem{theorem}{Theorem}
\newcounter{lemm}
\newtheorem{lemma}[lemm]{Lemma}
\newcounter{prop}
\newtheorem{proposition}[prop]{Proposition}
\newcounter{coro}
\newtheorem{corollary}[coro]{Corollary}

\newcounter{rema}
\newenvironment{remark}{\smallskip\noindent\stepcounter{rema}{\bf Remark \arabic{rema}}.}{\qed\smallskip}
\newcounter{exam}
\newenvironment{example}{\smallskip\noindent\stepcounter{exam}{\bf Example \arabic{exam}}.}{\qed\smallskip}

\begin{document}
\date{\today}
\vspace*{30mm}
\title{Five dimensional almost para-cosymplectic manifolds with contact Ricci potential}
\author{Piotr Dacko}

\email{{\tt piotrdacko{\char64}yahoo.com}}
\begin{abstract}
In the paper are studied  5-dimensional pseudo-Riemannian manifolds equipped with the 
almost para-contact structure which is an \
analogon of the almost contact structure in the Riemannian geometry. It is assumed that curvature and 
the structure affinor commute. These manifolds admits 
a compatible cosymplectic structure in the sense of P. Libermann.  There are found explicit expressions 
of the connection, the curvature, the Weyl curvature and the Ricci tensor. The structure affinor 
and the Ricci tensor define a closed 2-form called the Ricci form. 
In the paper are  classified manifolds 
with the $\eta$-Einstein Ricci tensor,  manifolds with the contact Ricci potential, which is a classical contact form, 
locally flat manifolds and there are described all connected, simply connected Lie groups  admitting 
left-invariant structure.  All these manifolds are in fact Walker spaces - manifolds with parallel 
isotropic distribution. However the emphasize is put on the structure, in 
consequence the Walker's classification theorem is not applied.   
\end{abstract}
\maketitle

\begin{center}
 \small{\it Dedicated to the Memory of my Mother}
\end{center}

\section{ Introduction.}

The paper \cite{Erdem} arose an interest in studying similar problems in the framework
of almost para-cosymplectic geometry. It was very soon clear that the geometry of an almost para-cosymplectic manifold is quite different than an almost cosymplectic manifold against some expectations. The most important difference of what we know today is that there is no Goldberg-Yano-like theorem: in \cite{GY} S. I. Goldberg and K. Yano proved that the condition 
\[
 R(X,Y)\varphi Z = \varphi R(X,Y)Z
\]
for an almost cosymplectic manifold implies
\[
\nabla\varphi = 0, 
\]
and the manifold is cosymplectic according to D. E. Blair theorem. The examples which appeared in \cite{Dtsu} showed that similar theorem  in the case of almost para-cosymplectic manifold cannot hold. So we may say that there are 'strong' para-cosymplectic manifolds satisfying
\[
 \nabla\varphi =0,
\]
and manifolds which satisfy only the weaker condition 
\[
 R(X,Y)\varphi Z = \varphi R(X,Y)Z,
\]
and these manifolds are called 'weakly para-cosymplectic'. It was very unexpected and attracted  much attention.

In \cite{Dtsu} it was obtained the local characterization of weakly para-cosymplectic manifolds with
para-K\"ahlerian leaves.
There are four classes classified by the properties a tensor field $A=-\nabla\xi$. 
As $A$ is self-adjoint (or in different terminology: symmetric) a symmetric form $g(AX,Y)$ gives rise to a notion
of elliptic, hyperbolic or parabolic manifold depending on a rank  $r(A) \leqslant 2$ and  index $i$:
\begin{enumerate}
 \item $r=1$, the manifold is parabolic,
 \item $r=2$
     \begin{enumerate}
         \item $i=0$ - hyperbolic, (or neutral)
         \item $i=2$ - elliptic, 
      \end{enumerate}
 \end{enumerate}
if the manifold has para-K\"ahler leaves then $r=0$ ($A=0$) is equivalent for the manifold to be para-cosymplectic.

We shortly summarize the contents of the paper. The Section \ref{secCurv} is general. The results provided are valid not only 
for 5-dimensional manifolds but in the general case. Here we establish recall the basic curvature identities, we enable the 
existence of the Ricci form and we prove that the Ricci form is closed. Also we prove that locally conformally flat 
weakly para-cosymplectic manifolds in dimensions $\geqslant 5$ are locally flat.

The Section \ref{secCon} is the main technical part of the paper. Here appears the notion of adopted frames and coframes, 
in this section are found the expressions for the Levi-Civita connection and Cartan's structure equations. Subsequently 
we obtain the form of the curvature, the Weyl curvature and the Ricci tensor which are fundamental. We introduce 
in the section the idea of algebraic scalar invariant as a particular function of the curvature coefficients.

In the Section \ref{secEta} we provide a local classification of $\eta$-Einstein manifolds, ie. manifolds with the Ricci tensor
of the form
\[
 Ric(X,Y) = \dfrac{r}{4}(g(X,Y)-\eta(X)\eta(Y)),
\]
$r$ denotes the scalar curvature and on the results of the Section \ref{secCurv}, $r$ have to be constant and it assumed $r\neq 0$. 

The Section \ref{secCont} is devoted to the local classification of the manifolds with contact Ricci potential. More precisely, 
 such manifold admits a contact form (non-unique) which exterior derivative is equal to the Ricci form $\rho$, and thus via the 
structure affinor $\varphi$, the Ricci form determines the Ricci tensor
\[
 \rho(\varphi X,Y)= Ric(X,Y),
\]
similarly as for example in the geometry of the K\"ahler manifolds.

In the Section \ref{secFlat} we classify locally flat manifolds. It is interesting to point out that the local flatness of the metric
does not imply integrability of the structure. Thus here are described elliptic or hyperbolic locally flat manifolds. Also here 
we provide an example of the locally flat manifold which admits a transitive group of automorphisms and in the fact there are at least 
two non-isomorphic maximal (with effective action) such groups.   

Finally in the last Section we classify all connected, simply connected Lie groups $\mathcal{G}$ which admits a left-invariant
structure $(\varphi,\xi,\eta,g)$, all the tensors are left-invariant, such that $(\mathcal{G},\varphi,\xi,\eta,g)$ is a weakly 
para-cosymplectic manifold with para-K\"ahler leaves. 

The research on the subject of this paper started several years ago. Partially, some results mainly form the Section \ref{secCurv},
from the last Section and some their consequences like the existence of the contact Ricci potential for $\eta$-Einstein manifolds, were presented at the seminar 
held at the University of Bari, Italy in June, 2011. On the occasion of the publishing of this paper I would like to express my gratitude 
to prof. Anna M. Pastore, prof. Maria Falcitelli for invitation and hospitality. Also I would like to thank heartily to my colleagues 
Giulia  Dileo and Vincenzo Saltarelli. Particularly to Giulia Dileo who was my guide during my visit. 
My visit was possible thank to financial support from the University of Bari, including the covering of 
all expenses like the airplane and accommodation. 

Also I would like to express my sincere gratitude to Danuta Konfederat. Thank to  her invaluable support and encouragement 
it was possible to finish this paper.

The References provided are far from being complete. The author in advance would like to express his gratitude to everyone who 
will help to complete the References pointing the papers which cover the topic of this paper or have other substantial relationships.

At some parts of this paper for testing purposes the author used ``Maxima'' CAS-algebra system version 5.29
\footnote{http://maxima.sourceforge.net}, which is a free GPL-based, open source application derived from LISP based ``MACSYMA'' created at MIT. The author found ``Maxima'' a useful replacement for all commercial products. ``Maxima'' is low-resources application: the author used it quite effectively on Celeron\texttrademark  1.2 GHz microprocessor with 512 MB of memory, on Slackware\texttrademark Linux v. 12.2.

\section{Preliminaries}
\subsection{Almost para-cosymplectic structures and manifolds}
A quadruple  $(\phi,\xi,\eta,g)$ of the tensor fields on a manifold $\mathcal{M}$, $\dim \mathcal{M}=2n+1 \geqslant 3$, such that
\begin{equation*}
\begin{array}{l}
\phi^2 = Id - \eta\otimes \xi,\quad \eta(\xi) = 1, \\
g(\phi X,\phi Y) = -g(X,Y) +\eta(X)\eta(Y), \quad \eta(X) = g(\xi,X),
 ×
\end{array}
\end{equation*}
is called an almost para-contact metric structure with hyperbolic metric \cite{Erdem}.
The tangent bundle  splits into a direct sum 
\[
 T\mathcal{M} = \lbrace\xi\rbrace\oplus\mathcal{D}_{-1}\oplus\mathcal{D}_{+1},
\]
where $\lbrace\xi\rbrace$ denotes a line vector bundle spanned by the non-vanishing vector field $\xi$ and $\mathcal{D}_{\pm 1}$ are vector bundles corresponding to the $-1$ and $+1$ eigenspaces of $\varphi$, $\dim \mathcal{D}_{-1} = \dim \mathcal{D}_{+1} = n$. The distributions $\mathcal{D}_{-1}$, $\mathcal{D}_{+1}$ are totally isotropic
\[
 g(X,Y)=0, \quad\; \text{if}\;X, Y \in \Gamma(\mathcal{D}_{-1}) \; \text{or}\; X, Y \in \Gamma(\mathcal{D}_{+1}).
\]
A tensor field
\[
 \varPhi(X,Y)=g(\phi X,Y),
\]
is a 2-form called the fundamental form of the manifold.

A manifold $\mathcal{M}$ equipped with an almost para-contact metric structure $(\varphi,\xi,\eta,g)$ is called  almost para-cosymplectic if the structure form $\eta$ and the fundamental form $\varPhi$ are both closed 
$$d\eta=0,\quad d\varPhi=0.$$
The annihilator  $\mathcal{A}: \eta = 0$  defines a completely integrable distribution (as $\eta$ is closed). 
A leaf $\mathcal{N}\in \mathcal{F}$ of the corresponding codimension one foliation  $\mathcal{F}$ in a natural way   
inherits an almost para-Hermitian structure $(P, H)$
\[
 P^2 = Id,\quad H(P\cdot,P\cdot) = -H(\cdot,\cdot).
\]
 Indeed, if $\sigma$ is an inclusion map
$\sigma: \mathcal{N} \subset \mathcal{M}$ then
\[
 \sigma_*(PX) = \varphi(\sigma_* X), \quad g(\sigma_* X,\sigma_* Y) = H(X,Y),
\]
for vector fields $X$, $Y$ tangent to $\mathcal{N}$. For a survey on almost para-Hermitian geometry we refer to \cite{CFG}.

An almost para-cosymplectic manifold $\mathcal{M}$ with $\nabla\varphi=0$ is called para-cosymplectic. 
In this case the manifold $\mathcal{M}$ locally is a metric product of an open interval $(-a,a)$  and a para-K\"ahler manifold. For a point $p\in \mathcal{M}$ there is  a neighborhood 
$\mathcal{U}$ and a local isometry 
\[
 : \mathcal{U} \rightarrow (-a,a)\times (\mathcal{N}_p \cap \mathcal{U})^\circ,
\]
where $(\mathcal{N}_p\cap \mathcal{U})^\circ$ is a connected component of $\mathcal{N}_p\cap \mathcal{U}$ containing $p$,
$\mathcal{N}_p$ is a leaf passing through $p$. This local description implies that all leaves of $\mathcal{F}$ are para-K\"ahler manifolds \cite{Dtsu}.

In general manifolds with para-K\"ahler leaves are characterized by the identity \cite{Dtsu}:
\begin{equation}
\label{nabfi}
(\nabla_X\varphi)Y = g(A\varphi X, Y)\xi - \eta(Y)A\varphi X, \quad A = -\nabla\xi.
\end{equation}
There are many examples where a manifold has para-K\"ahler leaves and is not para-cosymplectic. As a consequence of (\ref{nabfi}) we have the following curvature identity 
\[
 \begin{array}{rcl}
R(X,Y)\varphi Z - \varphi R(X,Y)Z &=& g(A\varphi X, Z)AY-g(A\varphi Y, Z)AX + g(AX,Z)A\varphi Y \\
&& - g(AY,Z)A\varphi X -g(R(X,Y)\xi,\varphi Z)\xi-\eta(Y)\varphi R(X,Y)\xi.
 ×
\end{array}
\]

In the contrary to the Riemannian case of almost cosymplectic structures \cite{GY} the condition 
\[
 [R(X,Y),\phi] = 0 
\]
does not imply that the manifold is para-cosymplectic. 
Manifolds satisfying $[R(X,Y),\phi]=0$ are called weakly para-cosymplectic.

\begin{theorem} {\rm(\cite{Dtsu})}
Let $\mathcal{M}$ be an almost para-cosymplectic manifold with para-K\"ahlerian leaves.  $\mathcal{M}$ is weakly para-cosymplectic if and only if the following conditions are satisfied:
\begin{equation}
\label{naba}
(\nabla_XA)Y=(\nabla_YA)X, \quad \text{\rm ($A$ is a Codazzi tensor)}
\end{equation}
and at any point $A$ is of the kind (a), (b) or (c)   
\[
 \begin{array}{l}
 \textrm{(a)}\;\; A = 0, \\[+4pt]
 \textrm{(b)}\;\; AX = \pm g(X,V)V, \varphi V = \pm V,    \\[+4pt]
 \textrm{(c)}\;\; AX = \varepsilon_1 g(X,V_1)V_1 + \varepsilon_2 g(X, V_2),  \quad g(V_1,V_2) = 0, 
                   \quad \varepsilon_1 = \varepsilon_2 = \pm 1,\\[+2pt]
            \;\;\;\;\;\;\;  \phi V_1 = -V_1 , \quad \phi V_2 = V_2,                
 ×
\end{array}
\]
\end{theorem}
At the points where $A=0$ the covariant derivative $\nabla\varphi$ vanishes. Independently 
of the algebraic form we  have always the identity $A^2 = 0$; at each point $A\neq 0$ induces a 2-step nilpotent endomorphism of tangent space. From the definition 
\[
  g(AX,Y) = -(\nabla_X\eta)(Y),
\]
hence the form $g(AX,Y)$ is symmetric; the properties of this form give rise to the classification: 
 $\mathcal{M}$ is called
\begin{itemize}
 \item[b)] parabolic $\iff$ $g(AX,Y) = \varepsilon \theta(X)\theta(Y)$, 
 \item[c)] elliptic  $\iff$ $g(AX,Y) = \varepsilon (\theta^1(X)\theta^1(Y)+\theta^2(X)\theta^2(Y))$, 
 \item[d)] hyperbolic $\iff$ $g(AX,Y) = \varepsilon (\theta^1(X)\theta^1(Y)-\theta^2(X)\theta^2(Y))$.
\end{itemize}
As we see the classification follows the pattern of the classification of the quadratic forms on a 2-plane. Here $\theta^1(X) =g (X,V_1)$, $\theta^2(X) = g(X,V_2)$ and $V_1$, $V_2$ are as in the above Theorem.

\subsection{The case of the dimension 3}
Here we provide a short overview of a geometry of 3-dimensional almost para-cosymplectic manifolds. Note that 
in the dimension 3 the leaves are always para-K\"ahler: arbitrary 3-dimensional manifold has para-K\"ahler leaves.
Those manifolds which are additionally weakly para-cosymplectic are classified and studied in \cite{DO}. These 
manifolds possess many interesting properties. For example the scalar curvature $r$ determines completely the curvature operator: for the Riemann curvature $R$ and the Ricci tensor $S$ we have   
\[
 \begin{array}{ll}
 R(X,Y) = (r/2)\varPhi(X,Y)\phi,  \quad & \text{(the curvature)} \\
 S(X,Y) = (r/2)(g(X,Y)-\eta(X)\eta(Y)), \quad & \text{(the Ricci tensor)} \\
\end{array}
\]
Due to the fact that $\mathcal{M}$ is three-dimensional there are no elliptic or hyperbolic points on $\mathcal{M}$: the rank of $A$, $r(A)$ is $\leqslant 1$.  If we assume that $\mathcal{M}$ is essentially weakly para-cosymplectic then 
$r(A)=1$ everywhere. Near each point $p\in \mathcal{M}$ there is an isotropic (null) vector field $V$ (determined up to the sign), such that 
\[
 AX = \varepsilon g(X,V)V, \quad \varepsilon =\pm 1.
\]
 The vector field $V$ is recurrent $\nabla_XV = \tau(X)V$ and  the  form $\tau$ determines the scalar curvature
\[
 r = d\beta(V), \quad \tau(X) = \beta g(X,V).
\]
A globally defined one-dimensional isotropic distribution spanned by  $V$  is isotropic and parallel. Thus $\mathcal{M}$ is a Walker space: the pseudo-Riemannian manifold with parallel isotropic distribution. Locally 
the structure of $\mathcal{M}$ is described with a use of a particular coordinates chart $(x,y,z)$
\[
 \begin{array}{l}
 \xi = \partial_z, \quad \eta = dz, \\[+4pt]
  \varphi\;\partial_x = \epsilon_2\;\partial_x, \quad \\ [+4pt]
\varphi\;\partial_y = 2\epsilon_2(b(x,y)-\epsilon_1z)\partial_x - \epsilon_2\;\partial_y, \quad \\ [+4pt]
g=\begin{pmatrix}
 0 & 1 & 0 \\
 1 & 2(b(x,y)-\epsilon_1z) & 0 \\
 0 & 0 & 1
  \end{pmatrix}
,\\[16pt]
 \epsilon_i = \pm 1,\quad i=1,2,
\end{array}
\] 
and the scalar curvature 
\[
    r(x,y,z) =2\partial^2_x b.
\]
A locally Riemannian homogeneous manifold $\mathcal{M}$ either is para-cosymplectic and leaves are para-K\"ahler surfaces of constant curvature,  or in the  case $\nabla\varphi \neq 0$,  the universal covering 
$\widetilde{\mathcal{M}}$ is a Lie group with a  left-invariant almost para-contact metric structure $(\widetilde\varphi,\widetilde\xi,\widetilde\eta,\widetilde g)$: $(\widetilde{ \mathcal{M}}, \widetilde\varphi,\widetilde\xi, \widetilde\eta, \widetilde g)$  and $(\mathcal{M},\varphi,\xi, \eta, g)$ are locally isomorphic as almost para-contact metric manifolds. The metric $g$ on $\mathcal{M}$ necessary is locally flat. Therefore the covering $(\widetilde{\mathcal{M}}, \widetilde g)$ is pseudo-Euclidean. The 
underlying  Lie group $\widetilde{\mathcal{M}}$ is either  Heisenberg group $\mathbb{H}^3$ or a unique non-unimodular Lie group $\mathbb{A}^3$ \cite{DO}. The other consequence is that an  essentially weakly para-cosymplectic manifold with non-zero curvature is irreducible as a pseudo-Riemannian manifold. Hence there are no weakly para-cosymplectic three-manifolds of a non-zero constant sectional curvature. Concerning the last statement note that in the dimensions $\geqslant 5$ we have a much stronger result \cite{Dtsu}: if $\mathcal{M}$ is arbitrary almost para-cosymplectic manifold, $\textrm{dim}(\mathcal{M})\geqslant 5$, of a constant sectional curvature $K$, then $K=0$ and $\mathcal{M}$ is locally 
flat para-cosymplectic manifold.

\section{Curvature identities. The Ricci form}
\label{secCurv}
In this section we establish basic curvatures identities for a weakly para-cosymplectic manifold, define a Ricci form and prove that the Ricci form is closed if the manifold has para-K\"ahler leaves. Moreover we shall prove that conformally flat weakly 
para-cosymplectic manifolds in dimensions $\geqslant 5$ are locally flat.
Throughout this section and further $R(X,Y)Z$ denotes the Riemann curvature operator of the Levi-Civita connection 
\[
 R(X,Y)Z = \nabla_X\nabla_YZ -\nabla_Y\nabla_XZ-\nabla_{[X,Y]}Z,
\]
$R(X,Y,Z,W)$ denotes the Riemann-Christoffel curvature 
\[
 R(X,Y,Z,W) = g(R(X,Y)Z,W),
\]
$Ric$ the Ricci tensor
\[
 Ric(X,Y) = Tr\lbrace Z\mapsto R(Z,X)Y\rbrace,
\]
and $Q$ the Ricci operator defined by the equation 
\[
     Ric(X,Y) = g(QX,Y).
\]
\begin{proposition}
\label{phi-symm}
 Let $\mathcal{M}$ be a weakly para-cosymplectic manifold. We have the following identities
\begin{equation}
 \begin{array}{l}
\varphi R(X,Y)Z =  R(X,Y)\varphi Z, \\
R(\varphi X,\varphi Y)Z =  -R(X,Y)Z, \\
R(\varphi X,Y)Z = - R(X, \varphi Y)Z, \\
R(X,Y)\xi = R(\xi,X)Y = 0, 
 ×
\end{array}
\end{equation}
for the Riemann-Christoffel curvature
\begin{equation}
 \begin{array}{l}
R(\varphi X, Y, Z, W) = R(X, Y, \varphi Z, W) = -R(X,\varphi Y, Z, W) = -R(X, Y, Z, \varphi W), \\
R(\varphi X, \varphi Y, Z, W) = R(X,Y,\varphi Z, \varphi W) = -R(X,Y,Z,W),\\
R(\xi,X,Y,Z) = R(X,\xi,Y,Z) = R(X,Y,\xi,Z) = R(X,Y,Z,\xi) = 0,
 ×
\end{array}
\end{equation}
and for the Ricci tensor
\begin{equation}
 \begin{array}{l}
Ric(\varphi X, Y) = - Ric( X,\varphi Y), \\
Ric(\varphi X, \varphi Y) = -Ric(X,Y),\\
Ric(X,\xi) = 0,
 ×
\end{array}
\end{equation}
in all the above formulas $X$, $Y$, $Z$, $W$ are arbitrary vector fields on $\mathcal{M}$.
\end{proposition}
\begin{proof}
  From the definition
\[
 \varphi R(X,Y)Z = R(X,Y)\varphi Z,
\]
setting $Z=\xi$ we find $\varphi R(X,Y)\xi = 0$ hence
\[
 0 = \varphi^2 R(X,Y)\xi = R(X,Y)\xi -\eta(R(X,Y)\xi)\xi = R(X,Y)\xi, 
\]
note that $\eta(R(X,Y)\xi) = g(R(X,Y)\xi,\xi) = R(X,Y,\xi,\xi)=0$. Projecting onto $Z$ 
\[
 0 = g(R(X,Y)\xi, Z) = R(X,Y,\xi,Z), 
\]
the symmetries of the Riemann-Christoffel curvature follow
\[
 0 = R(\xi, Z, X,Y) = R(Z, \xi, X,Y) = R(X, Y, Z ,\xi ) = 0,
\]
in  consequence $R(\xi,Z)X = 0$. Now
\[
 \varphi R(X,Y)\varphi Z = R(X,Y)\varphi^2 Z = R(X,Y)Z -\eta(Z)R(X,Y)\xi = R(X,Y)Z, 
\]
projecting both hands onto $W$ we find
\[
   R(X,Y,\varphi Z, \varphi W) = - g(\varphi R(X,Y)\varphi  Z,W) = - R(X,Y,Z,W),
\]
 we substitute $W=\varphi W'$ in the last expression
\[
 R(X, Y, \varphi Z, W') =  - R(X, Y, Z , \varphi W'),
\]
\end{proof}
\begin{corollary}
 A 2-form 
\[
 \rho(X,Y) = Ric(\varphi X, Y),
\]
will be called the Ricci form
\end{corollary}
\begin{proposition}
 For a weakly para-cosymplectic manifold $\mathcal{M}$ with para-K\"ahler leaves the Ricci form is closed
\[
 d\rho =0.
\]
\end{proposition}
\begin{proof}
 The idea of the proof is to use the identity
\[
 \sum\limits_{X,Y,Z} (\nabla_X\rho)(Y,Z) = 3 d\rho(X,Y,Z),
\]
here on  the left side $\sum\limits_{X,Y,Z}$ denotes the cyclic sum over the entries $X$,$Y$,$Z$.
We have 
\[
 \rho(Y,Z) = \textrm{Tr}\{X \mapsto R(X,\varphi Y)Z\} = \textrm{Tr}\{ X \mapsto -R(\varphi X, Y)Z \}, 
\]
and
\[
 (\nabla_W\rho)(Y,Z) = \textrm{Tr}\{X \mapsto -(\nabla_WR)(\varphi X,Y)Z - R((\nabla_W\varphi)X,Y)Z\}. 
\]
By (\ref{nabfi})
$
 R((\nabla_W\varphi)X,Y)Z = -\eta(X)R(A\varphi W,Y)Z
$
and
\[
 \textrm{Tr}\{X \mapsto \eta(X)R(A\varphi W,Y)Z \} =  \eta(R(A\varphi W, Y)Z) = 0,
\]
therefore for the covariant derivative $\nabla\rho$ we obtain 
\[
 (\nabla_W\rho)(Y,Z) = \textrm{Tr}\{X \mapsto - (\nabla_WR)(\varphi X, Y)Z\}.
\]

The first Bianchi identity applied to $\nabla_W R$ 
\[
 (\nabla_WR)(\varphi X, Y)Z = -(\nabla_WR)(Y,Z)\varphi X - (\nabla_WR)(Z,\varphi X)Y,
\]
and the second Bianchi identity follow
\[
 \sum\limits_{W,Y,Z}(\nabla_WR)(\varphi X, Y)Z = \sum\limits_{W, Z,Y}(\nabla_WR)(\varphi X,Z)Y, \qquad \text{(the cyclic sums)}
\]
taking the traces $\textrm{Tr}\{X \mapsto \cdot\}$ of respectively left and right hands we find
\[
 d\rho(W,Y,Z) = d\rho(W,Z,Y),
\]
 as $Y$ and $Z$ are arbitrary $d\rho = 0$ identically.
 \end{proof}

An almost para-cosymplectic manifold $\mathcal{M}$ is called $\eta$-Einstein if the Ricci 
tensor  satisfies
\[
 Ric(X,Y) = \dfrac{r}{2n}( g(X,Y)-\eta(X)\eta(Y)), 
\]
$r$ is the scalar curvature $r$, $dim(\mathcal{M}) = 2n+1$.

\begin{theorem}
\label{scEtaEin}
For  $\eta$-Einstein, weakly para-cosymplectic manifold $\mathcal{M}$ with para-K\"ahler leaves and $dim(\mathcal{M}) \geqslant 5$, the scalar curvature is constant
\[
   r = r_0.
\]
If the Ricci tensor is parallel with respect to the Levi-Civita connection 
\[
 \nabla Ric =0,
\]
then $\mathcal{M}$ is para-cosymplectic or Ricci-flat $Ric =0$.
\end{theorem}
 We start the proof from the following algebraic lemma
\begin{lemma}
 Let $\omega$ be a 2-form of maximall rank at any point. If there is a non-zero 1-form 
$\tau$ such that
\[
 \tau\wedge\omega = 0,
\]
 then $dim(\mathcal{M}) \leqslant 3$. 
\end{lemma}
{\it Proof of the Lemma.} As $\tau\neq 0$ there is a vector field $V$ such that $\tau(V) = \iota_V\tau \neq 0$. The inner product
\[
 \iota_V(\tau\wedge\omega) = (\iota_V\tau)\omega-\tau \wedge\iota_V\omega,
\]
follows $\omega =  \tau\wedge\beta$ for a 1-form $\beta$ ( precisely $\beta$ is defined at points where $\iota_V\tau \neq 0$) and in a consequence $\omega^2 = 0$. As $\omega$ has maximall rank $dim(\mathcal{M}) < 4$.
\begin{proof}
By the assumption 
\[
 \rho(X,Y) = Ric(\varphi X,Y) = \dfrac{r}{2n}\varPhi(X,Y)
\]
 and as $\varPhi$ is closed
\[
 0 = d\rho = \dfrac{1}{2n}dr\wedge\varPhi.
\]
The fundamental form $\varPhi$ has maximall rank. If $dr \neq 0$  applying the Lemma we obtain $dim(\mathcal{M}) < 4$: the contradiction. Thus $dr$ must vanish which implies $r=r_0= const$. 
Assuming $\nabla Ric=0$ we obtain the following equivalent condition
\[
 0 = \dfrac{r}{2n}(g(AW,X)\eta(Y)+g(AW,Y)\eta(X)),
\]
setting $Y=\xi$ we find that $r g(AW,X) =0$ for arbitrary vector fields $W$, $X$, thus $r=0$ or $A = 0$ In the former case the manifold is Ricci-flat in the latter case the vanishing of $A$ follows $\nabla\varphi=0$ - the manifold is para-cosymplectic. 
\end{proof}

\begin{theorem}
 A locally conformally flat weakly para-cosymplectic manifold $\mathcal{M}$, $dim(\mathcal{M})=2n+1 \geqslant 5$ is
locally flat $R = 0$.  
\end{theorem}
\begin{proof}
 We start from the well-known decomposition for the Riemann-Christoffel curvature
\[
\begin{array}{rcl}
R(X,Y,Z,W) &=& C(X,Y,Z,W) + \dfrac{1}{2n-1}( Ric(Y,Z)g(X,W)-Ric(Y,W)g(X,Z) \\ [+8pt]
           & & \qquad\quad + g(Y,Z)Ric(X,W) -g(Y,W)Ric(X,Z)) \\[+6pt]
           & & - \dfrac{r}{2n(2n-1)}(g(Y,Z)g(X,W)-g(Y,W)g(X,Z)),
\end{array}
\]
$C(X,Y,Z,W)$ denotes the Weyl curvature. Under assumption that $\mathcal{M}$ is locally conformally flat $C=0$.
For weakly para-cosymplectic manifold $R(X,Y,\xi,W)=0$ identically which yields that the vector field $\xi$ 
lies in the kernel of the Ricci tensor $Ric(\xi,\cdot)=0$. Thus
\[
0= R(X,Y,\xi,W) = \eta(Z)(Ric(X,W)-\frac{r}{2n}g(X,W))  
   -\eta(W)(Ric(X,Z)-\frac{r}{2n}g(X,Z)),
\]
setting $Z=\xi$ we obtain $Ric(X,W)=\frac{r}{2n×}(g(X,W)-\eta(X)\eta(W))$ that is $\mathcal{M}$ is $\eta$-Einstein.
As $\mathcal{M}$ is $\eta$-Einstein and conformally flat we find 
\[
 R(\varphi X,\varphi Y, Z, W) = \dfrac{r}{2n(2n-1)}(\varPhi(Y,Z)\varPhi(X,W)-\varPhi(Y,W)\varPhi(X,Z)),
\]
$\varPhi$ is the fundamental form, from the other hand by the Proposition \ref{phi-symm}, eq. (4)
\[
 R(X,Y,Z,W) = -\dfrac{r}{2n(2n-1)}(\varPhi(Y,Z)\varPhi(X,W)-\varPhi(Y,W)\varPhi(X,Z)),
\]
now applying the first Bianchi identity follows
\[
 r\,\varPhi\wedge\varPhi = 0,
\]
 hence $r=0$ and in consequence $R=0$.  
\end{proof}

\section{Connection and curvature of elliptic and hyperbolic five-manifolds}
\label{secCon}
In this section we provide the basic formulas for the connection and the curvature in both cases: elliptic 
and hyperbolic. The main tool to reach these goals is a set of a particular frames we are going to describe. 
\subsection{Adopted frames}
Before proceeding further we need  the following general 
\begin{proposition}
\label{locfrgen}
 Let $\mathcal{M}$ be an almost para-cosymplectic manifold, $dim(\mathcal{M})=2n+1 \geqslant 3$ . Near each point $p\in \mathcal{M}$ there is 
a local frame of vector fields $(E_0, E_1, \ldots, E_n, E_{n+1},\ldots, E_{2n})$ such that 
\[
\begin{array}{l}
 E_0 = \xi, \quad \varphi E_i = E_i, \quad \varphi E_{i+n} = -E_{i+n},  \quad i = 1,\ldots,n, \\[+4pt]
 g(E_0,E_i) = g(E_0,E_{i+n}) = 0, \quad g(E_i, E_{i+n}) = 1, \quad i = 1,\ldots,n,
\end{array}
\]
and the other coefficients of the metric tensor vanish.
\end{proposition}
\begin{proof}
 Let assume that we already have constructed $(E_0, E_1, \ldots, E_k, E_{1+n}, \ldots, E_{k+n})$, $k \leqslant n$. Obviously we may assume $k<n$. Let denote by $\mathcal{D}_k$ the 
(2k+1)-dimensional local distribution spanned by $(E_0, E_1, \ldots, E_k, E_{1+n}, \ldots, E_{k+n})$. The distribution $\mathcal{D}_k$ is non-degenerate: the restriction $g|_{\mathcal{D}_k}$ of the metric tensor is non-degenerate. Thus there is properly defined orthogonal complement $\mathcal{D}_k^\perp$; we have  the splitting of the tangent bundle restricted to a small neighborhood $\mathcal{U}_p$ of a some point $p$
\[
 T\mathcal{U}_p = \mathcal{D}_k\oplus \mathcal{D}_k^\perp,
\]
the orthogonal complement $\mathcal{D}_k^\perp$ is also non-degenerate: so there is a non-null local section $V\in \Gamma(\mathcal{D}_k^\perp)$, we may assume  $g(V,V) = \varepsilon$, $\varepsilon = \pm 1$. Now let (note $\eta(\mathcal{D}_k^\perp)=0$)
\[
 E_{k+1} = \frac{1}{\varepsilon\sqrt{2}}(Id+\varphi)V, \quad E_{k+1+n} = \frac{1}{\sqrt{2}}(Id-\varphi)V, 
\]
hence 
$$(Id-\varphi)E_{k+1} = \frac{1}{\varepsilon\sqrt{2}×}(Id^2-\varphi^2)V=0, \quad (Id+\varphi)E_{k+1+n}=0,
$$
and
\[
g(E_{k+1}, E_{k+1+n}) = 
\dfrac{1}{2\varepsilon }(g(V,V) -g(V,\varphi V) +g(\varphi V, V) - g(\varphi V , \varphi V)) = 1.  
\]
The frame $(E_0, E_1, \ldots, E_k, E_{1+n}, \ldots, E_{k+n}) \cup (E_{k+1},E_{k+1+n})$ satisfies the requirements of the Proposition. The construction starts from $E_0=\xi$.
\end{proof}

\begin{theorem}
\label{Ashape}
 Let $\mathcal{M}$ be an elliptic or hyperbolic weakly para-cosymplectic manifold with para-K\"ahlerian 
leaves. Then near each point there are non-vanishing vector fields $V_1$, $V_2$ such that 
\[
\begin{array}{l}
AX = \varepsilon( g(X,V_1)V_1-g(X,V_2)V_2), \quad \textrm{- the hyperbolic case,}\\[+4pt]
AX = \varepsilon( g(X,V_1)V_1+g(X,V_2)V_2), \quad \textrm{- the elliptic case,}
 ×
\end{array}
\]
$\varphi V_1 = -V_1$, $\varphi V_2 = V_2$, the vector fields $V_i$, $i=1,2$ are determined uniquely up to the sign $V_i \mapsto -V_i$.  
\end{theorem}
\begin{proof}
 We adopt here the main idea of the proof of the Theorem {\bf 5} in \cite{Dtsu}. For a point 
$p\in\mathcal{M}$ let $(E_0, E_i,E_{i+n})$, $i=1,\ldots,n$ be a local frame constructed in the Proposition {\bf \ref{locfrgen}.} We set
\[
 c_{ij} = g(AE_i, E_j), \quad c_{(i+n)(j+n)} = g(AE_{i+n}, E_{j+n}), \quad i,j = 1,\ldots,n
\]
note ($\varphi A=-A\varphi$)
\[
 c_{i(j+n)} = g(AE_i, E_{j+n}) = - g(\varphi AE_i, \varphi E_{j+n}) = - c_{i(j+n)},
\]
hence $c_{i(j+n)} = 0$ identically  and 
$$
c_{0i} = c_{0(i+n)}= g(AE_0,E_i)=g(AE_0, E_{i+n}) = 0.
$$
 If  
$X = \sum\limits_{i= 0}^{2n}X^iE_i$, $Y= \sum\limits_{j=0}^{2n}Y^jE_j$ are local vector fields defined near the point $p$ then 
\[
 g(AX,Y) = \sum\limits_{i,j=1}^n c_{ij}X^iY^j+ \sum\limits_{i,j=1}^n c_{(i+n)(j+n)}X^{(i+n)}Y^{(j+n)},
\]
by the assumption $\textrm{rank}(g(A\cdot, \cdot))=2$ which implies (cf. \cite{Dtsu}) that the square, symmetric $n\times n$ matrices $[c_{ij}]_{i=1,j=1}^{n,n}$, $[c_{(i+n)(j+n)}]_{i=1,j=1}^{n,n}$ both have  rank $=1$.
Therefore there are smooth functions $d_1,\ldots,d_n$, $h_1,\ldots,h_n$ defined on a possibly smaller neighborhood of $p$, such that 
\[
 c_{ij} = \varepsilon_1 d_i d_j, \quad c_{(i+n)(j+n)} = \varepsilon_2 h_ih_j, \quad \varepsilon_1 = \varepsilon_2 = \pm 1,
\]
the functions $d_i$, $h_i$ are determined uniquely up to the change of the sign 
$(d_1,\ldots,d_n) \mapsto (-d_1,\ldots, -d_n)$,  $(h_1,\ldots,h_n) \mapsto (-h_1,\ldots,-h_n)$. Let define
\[
 \theta^1(X) =\sum\limits_{i=1}^n d_iX^i, \quad \theta^2(X) = \sum\limits_{i=1}^n h_i X^{i+n}, 
\quad V_1 = \sum\limits_{i=1}^n d_i E_{i+n}, \quad V_2 = \sum\limits_{i=1}^n h_i E_i,
\]
then  $\varphi V_1 = -V_1$, $\varphi V_2 = V_2$, $\theta^i(X) = g(X,V_i)$, $i=1,2$ and 
\[
 g(AX,Y) = \varepsilon_1\theta^1(X)\theta^1(Y)+\varepsilon_2\theta^2(X)\theta^2(Y),
\]
depending on the sign $\sigma=\varepsilon_1\varepsilon_2$ we have: the elliptic case $\sigma=1$, the 
hyperbolic case $\sigma = -1$
\[
 AX = \varepsilon_1 (\theta^1(X)V_1 + \sigma \theta^2(X)V_2).
\]
\end{proof}
\begin{remark}
\label{xorient}
 In the above formula $\varepsilon_1$ can be arbitrary $-1$ or $+1$. To simplify the notation we 
introduce a concept of $\xi$-orientability. Let consider a deformation of the structure $(\varphi, \xi, \eta)\mapsto (\varphi', \xi', \eta')$ defined by
\[
 \varphi' = \varphi, \quad \xi' = -\xi, \quad \eta' = -\eta.
\]
 The forms 
\[
  \eta'\wedge\varPhi'^n,\quad \eta\wedge\varPhi^n,          
\]
determine the opposite orientations on the manifold. The deformation follows $A \mapsto A' = -A$. Thus if necessary changing the direction of $\xi$ we can assure that $\varepsilon_1 = 1$. 
\end{remark}

Let $(\mathcal{M},\phi,\xi,\eta, g)$ be a 5-di\-men\-sional weakly para-cosymplectic manifold with para-K\"ahlerian leaves.  Assume that $\mathcal{M}$ is non-degenerate, we mean by that the tensor field $A$ is of maximall rank. Thus $\mathcal{M}$ is elliptic or hyperbolic: here we assume that at each point.  So near each point there are vector fields $V_1$, $V_2$ and
\begin{equation}
AX = g(X, V_1)V_1 +\sigma g(X, V_2)V_2, \quad \phi V_1 = -V_1, \quad
 \phi V_2 = V_2, \quad g(V_1,V_2) = 0,
\end{equation}
and $\sigma= \pm 1$.
Let extend $(\xi,V_1,V_2)$ to a local frame $(\xi, V_1, V_2, V_3, V_4)$
in the way that $\varphi V_3 = V_3$,  $\varphi V_4 =  -V_4$,  
 and the only non-zero coefficients of the metric  are
\begin{equation}  
g(\xi,\xi) = g(V_1, V_3) = g(V_2, V_4) = 1.
\end{equation}
For example we can take 
\[
 V_3 = \dfrac{U+\varphi U}{2}, \quad V_4 = \dfrac{W-\varphi W}{2},
\]
where $U$ and $W$ are local vector fields and such that
\[
 g(U,V_1) = 1, \quad g(W,V_2)= 1, \quad \eta(U)=\eta(W) = 0.
\]
If the coefficient $c=g(V_3,V_4) \neq 0$ we modify $V'_3 = V_3- cV_2$. Then  $(\xi, V_1,V_2,V'_3,V_4)$ is the frame with the required properties.  Such frames will be called adopted frames. An extension of the vector fields $(\xi, V_1,V_2)$ to an adopted frame
$(\xi, V_1,V_2,V_3,V_4)$ is in general non-unique. However if $(\eta, V_1, V_2, V'_3, V_4')$ is another adopted local frame then  there is a function $\alpha$ such that
\begin{equation}
V'_3 = V_3 + \alpha V_2, \quad V_4' = V_4 -\alpha V_1.
\end{equation}

The next Proposition describes locally the Levi-Civita connection. 

\begin{proposition}
\label{lc} 
Let $(\xi, V_1,V_2,V_3,V_4)$ be an adopted frame defined near a point $p$ on a neighborhood $\mathcal{U}_p$. There 
are local 1-forms $\tau_1$, $\tau_2$ and $\omega$ defined on $\mathcal{U}_p$ such that 
\begin{equation}
\label{lcf}	
\begin{array}{l}
\\[-1pt]
(a) \quad\begin{matrix}
\nabla_X\xi = -AX = -g(X, V_1) V_1 -\sigma g(X, V_2) V_2, \quad 
\end{matrix}
 \\[+8pt] 
(b)\quad
\begin{matrix}
\nabla_X V_1 = \tau_1(X)V_1,\quad  \tau_1=\alpha_1 g(X,V_1),  \\[+2pt]
\nabla_X V_2 = \tau_2(X) V_2, \quad \tau_2= \alpha_2 g(X,V_2), 
\end{matrix}
\\[+16pt]
(c)\quad
\begin{matrix}
\nabla_X V_3  = g(X, V_1) \xi + \omega(X) V_2 - \tau_1(X)V_3 , \\[+2pt]
\nabla_X V_4 = \sigma g(X, V_2) \xi - \omega(X) V_1 - \tau_2(X) V_4,
\end{matrix}
\\[+4pt]
\end{array}
\end{equation}
 $\alpha_1$,  $\alpha_2$ are functions on $\mathcal{M}$ and $\sigma=\pm 1$. 

\end{proposition}
\begin{lemma} We have the identities
\label{Anull}
\begin{equation}
\label{nkera}
A(\nabla_XA)Y = (\nabla_XA)AY= (\nabla_{AY}A)X=0,
\end{equation}
for arbitrary vector fields $X$, $Y$.
\end{lemma}
\begin{proof}
Let define $l(X,Y,Z) =g((\nabla_XA)Y,AZ)$.
 By  $(\nabla_X A)Y = (\nabla_Y A)X$  
\begin{equation}
\label{lsym}
 l(X,Y,Z) = l(Y, X, Z).
\end{equation}
From  $A^2=0$
\begin{equation}
\label{diffa}
 A(\nabla_XA)Y + (\nabla_XA)AY =  (\nabla_X A^2)Y =0,
\end{equation}
also note that $g((\nabla_X A)Y,Z) = g(Y, (\nabla_XA)Z)$ as $A$ is self-adjoint, hence
\begin{equation}
\label{lanti}
 l(X,Y, Z)  =   -g(AY, (\nabla_XA)Z) =  -l(X,Z,Y), 
\end{equation}
therefore (\ref{lsym}) and (\ref{lanti}) together imply $l = 0$ identically. In a consequence
for arbitrary $X$, $Y$
\[
0 = A(\nabla_X A)Y= (\nabla_X A)AY =  (\nabla_{AY}A)X,
\]
where in the last equation we again used the fact that $A$ is a Codazzi tensor. 
\end{proof}
\begin{proof}(Of the Proposition)

 \noindent {\bf Part (a).} 
 The equation $\nabla\xi = -A$ comes from the definition of $A$; the remaining part is the contents of the Theorem {\bf \ref{Ashape}} for this particular case.

\noindent {\bf Part (b).} From
 \begin{equation}
\begin{array}{l}
(\nabla_X\varphi)V_1 = g(A\varphi X , V_1) \xi - \eta(V_1) A\varphi X = 0, \\
(\nabla_X\varphi)V_2 = g(A\varphi X, V_2) \xi - \eta(V_2)A\varphi X =0,
 ×
\end{array}
\end{equation}
we obtain ($\varphi V_1 = -V_1, \varphi V_2= V_2$)
$$\varphi\nabla_X V_1 = - \nabla_X V_1, \qquad \varphi \nabla_X V_2 = \nabla_XV_2, $$
thus $\nabla_XV_1$ is a local field of $-1$-eigenvectors, $\nabla_XV_2$ is a local field of $+1$-eigenvectors and hence $\nabla_XV_1$ must be a linear combination of $V_1$ and $V_4$, while $\nabla_XV_2$ a linear combination of $V_2$, $V_3$  
$$
\begin{array}{l}
\nabla_XV_1 = \tau_1(X) V_1 + \kappa_1(X)V_4,\\
\nabla_XV_2 = \tau_2(X) V_2 + \kappa_2(X)V_3
 ×
\end{array}
$$
By the Lemma {\bf \ref{Anull}} (note also $V_1 = AV_3$, $\sigma V_2=AV_4$, $A^2=0$) 
$$
\begin{array}{l}
0 = (\nabla_XA)AV_3 = -A\nabla_XV_1= - \kappa_1(X)AV_4 = -\sigma\kappa_1(X)V_2, \\
0 = (\nabla_XA)AV_4 = -\sigma A\nabla_XV_2 = -\sigma \kappa_2(X) AV_3 = -\sigma\kappa_2(X)V_1
 ×
\end{array}
$$ 
hence as $X$ is arbitrary $\kappa_1 = \kappa_2 = 0$ identically. Now computing $\nabla A$ we find
\begin{equation}
(\nabla_XA)Y = 2\tau_1(X)g(Y,V_1)V_1 +2\sigma\tau_2(X)g(Y,V_2)V_2,
\end{equation}
and by $(\nabla_XA)Y =(\nabla_YA)X$ we have
\begin{equation}
\begin{array}{l}
\tau_1(X)g(Y,V_1) =\tau_1(Y)g(X,V_1), \\
  \tau_2(X)g(Y,V_2) =\tau_2(Y)g(X, V_2) ,
 ×
\end{array}
\end{equation}
 in a consequence 
$$
\begin{array}{c}
\tau_1(X)=\alpha_1g(X, V_1), \\
\tau_2(X)= \alpha_2g(X, V_2),  
 ×
\end{array}
$$
for functions $\alpha_1$, $\alpha_2$. 

\noindent {\bf Part (c).} Let 
\begin{equation}
 \nabla_XV_3 = \alpha^0(X) \xi + \alpha^1(X)V_1 + \alpha^2(X) V_2 + \alpha^3(X)V_3 + \alpha^4(X)V_4,
\end{equation}
with the help of what we have just proved we find 
\begin{equation}
 \begin{array}{l}
  \alpha^0(X) = g(\nabla_XV_3, \xi) = g(V_3, AX) = g(AV_3, X) = g(V_1, X), \\[+2pt]
   \alpha^1(X) = g(\nabla_XV_3,V_3) = 0, \\[+2pt]
  \alpha^2(X) = g(\nabla_XV_3,V_4) = \omega(X), \\[+2pt]
  \alpha^3(X) = g(\nabla_XV_3, V_1)  = - g(\nabla_XV_1, V_3) = -\tau_1(X), \\ [+2 pt]
  \alpha^4(X)= g(\nabla_XV_3, V_2)  = - g(\nabla_XV_2, V_3) = -\tau_2(X)g(V_2, V_3) = 0,
 \end{array}
\end{equation}
and we have to provide similar computations for $\nabla_XV_4$.
\end{proof}

\begin{corollary}
\label{dt12=0}
 The forms $\theta^1(X) = g(X, V_1)$,  $\theta^2= g(X, V_2)$ are  closed forms, $d\theta^1=d\theta^2=0$. 
\end{corollary}
\begin{proof}
It is enough to proof that the derivatives $\nabla\theta^1$ and $\nabla\theta^2$ are symmetric. But 
(cf. \ref{lcf}(b)) 
\[
 (\nabla_X\theta^1)(Y) = g(Y,\nabla_X V_1) = \tau_1(X)g(Y,V_1) = \alpha_1 g(X,V_1)g(Y,V_1),
\]
and similar computations for $\nabla\theta^2$.
\end{proof}

\subsection{The gauge transforms.}
In this short part we discuss how the objects which appears in the Levi-Civita connection in the Proposition 
{\bf\ref{lc}} the  forms $\tau_1$, $\tau_2$ and the form $\omega$ are related to the particular choice 
of the adopted frame. In other words we are interested in how these objects change under a gauge transform
\[
 \rho=(\xi, V_1, V_2, V_3, V_4) \mapsto \rho'=(\xi, V_1',V_2',V_3',V_4'),
\]
that is passing from one adopted frame to another such one. Above we have noticed that the adopted frames are related to each other through a function
\begin{equation}
\label{gtran}
 \begin{array}{l}
V_1' = \pm V_1, \quad V_2' = \pm V_2, \\
V_3' = V_3 +\alpha V_2, \quad V_4' = V_4 -\alpha V_1,
 ×
\end{array} 
\end{equation}
$\alpha$ is a smooth function defined on a common neighborhood where both adopted frames $\rho$, $\rho'$ exist.

\begin{proposition}
Let  two adopted local frames be defined on $\mathcal{U}$ and $\mathcal{U'}$ resp. Let 
$\tau_1$, $\tau_2$, $\omega$ are corresponding forms on $\mathcal{U}$
\[
 \nabla V_1 = \tau_1\otimes V_1,\quad \nabla V_2 = \tau_2\otimes V_2, \quad \omega(X) =g (\nabla_XV_3,V_4),
\]
 and $\tau_1'$, $\tau_2'$, $\omega'$ their counterparts on $\mathcal{U'}$
\[
 \nabla V_1' = \tau_1'\otimes V_1', \quad \nabla V_2' = \tau_2'\otimes V_2', \quad \omega'(X) = g(\nabla_XV_3',V_4').
\]
If $\mathcal{U}\cap\mathcal{U'}\neq \emptyset$ then on $\mathcal{U}\cap\mathcal{U'}$ 
\begin{equation}
\label{gauge}
 \tau_1' = \tau_1, \quad \tau_2' = \tau_2, \quad \omega' = \omega+\alpha(\tau_1+\tau_2)+d\alpha,
\end{equation}
\end{proposition}
\begin{proof}
 The first two equations are rather obvious, for the third
\[
\begin{array}{rcl}
 \omega'(X) &=& g(\nabla_XV_3',V_4') = g(\nabla_XV_3,V_4) +d\alpha(X)g(V_2,V_4)  
  + \alpha g(\nabla_XV_2, V_4)  \\ 
& & - \; \alpha g(\nabla_XV_3,V_1) \\
&=& \omega(X) +\alpha(\tau_1(X)+\tau_2(X)) +d\alpha(X),
\end{array}
\]
in the last step we applied the formulas for the connection from the Proposition {\bf \ref{lc}.}
\end{proof}
We have two important corollaries: the families of local forms $\tau_1$ and $\tau_2$ give rise to 
a properly globally defined forms - the local forms coincide on the intersections of their domains of the definition. The second corollary consider  properties of the family of local forms $\omega$.
\begin{corollary}
\label{pKDom}
 Let define 
\[
 D\omega = d\omega - \omega\wedge(\tau_1+\tau_2),
\]
then 
\[
 D\omega' = D\omega + \alpha (d\tau_1+d\tau_2),
\]
here $\alpha$ is a local function which defines the gauge transform.
\end{corollary}
\begin{proof}
 According to (\ref{gauge})
\[
 d\omega' = d\omega + d\alpha\wedge(\tau_1+\tau_2) + \alpha (d\tau_1+d\tau_2),
\]
and
\[
 \omega'\wedge(\tau_1'+\tau_2') = \omega\wedge(\tau_1+\tau_2) +d\alpha\wedge(\tau_1+\tau_2),
\]
hence
\[
 D\omega' = d\omega'-\omega'\wedge(\tau_1'+\tau_2') = D\omega + \alpha (d\tau_1+d\tau_2).
\]
\end{proof}

\subsection{The connection form and the curvature.}
Let $(V_0= \xi, V_1,\ldots,V_4)$ be a local adopted frame and $(\theta^0=\eta, \theta^1, \ldots, \theta^4)$ be a 
metric-dual coframe defined by 
\[
 \theta^i(X) = g(X,V_i), \quad i=0,\ldots,4,
\]
using Berger's musical notations we may write
\[
 \theta^i = V_i^\flat, \quad \textrm{or} \quad \theta^{i\sharp} = V_i,
\]
and this notation comes from classical idea of ``rising'' and ``lowering'' indices with the help of the metric tensor. A dual coframe $(\delta^0,\delta^1, \ldots, \delta^4)$ is defined as usually by the Kronecker's $\delta$
\[
 \delta^i(V_j) = \delta^i_j,
\]
the metric-dual and the dual coframes are related by 
$\theta^i(X) = \sum_j\delta^j(X)\theta^i(V_j)=\sum_j\delta^j(X)g_{ij}$, $g_{ij}=g(V_i,V_j)$, so
\[
 \theta^i = \sum\limits_j g_{ij}\delta^j, \quad \delta^j = \sum\limits_i g^{ji}\theta^i,
\]
$[g^{ij}]$ is the inverse matrix of $[g_{ij}]=[g(V_i,V_j)]$.
A local connection form is a $5\times 5$ matrix of local 1-forms $\Omega = (\omega_i^j)$, $i,j=0,\ldots,4$  defined by the conditions
\begin{equation}
\label{nabvi}
 \nabla V_i = \sum\limits_{k=0}^4 \omega_i^k\otimes V_k, \quad i=0,\ldots,4, 
\end{equation}
$g(\nabla_XV_i,V_j) = -g(V_i,\nabla_XV_j)$ implies
\[
 \omega_i^kg_{kj} = - \omega_j^kg_{ik}.
\]
 Directly from the Proposition {\bf \ref{lc}} we have 

\begin{equation}
\label{OmegaMat}
 \Omega = \begin{pmatrix}
 0 & 0 & 0 & \theta^1 & \sigma\theta^2 \\
-\theta^1 & \tau_1 &  0 & 0 & -\omega \\
-\sigma\theta^2 & 0 & \tau_2 & \omega & 0 \\
0 & 0 & 0 & -\tau_1 & 0 \\
0 & 0 & 0 & 0 & -\tau_2 
          \end{pmatrix}. 
\end{equation}
If $R(X,Y)Z$ denotes the Riemann curvature operator
\[
 R(X,Y)Z  = [\nabla_X,\nabla_Y]Z - \nabla_{[X,Y]}Z,
\]
we define local curvature forms $C_i^j$ by
\[
 R(X,Y)V_i = \sum\limits_{j=0}^4 C_i^j(X,Y)V_j,
\]
introducing a curvature matrix $C(X,Y) = [C_i^j(X,Y)]$ we have
\[
 C(X,Y) = 2d\Omega(X,Y) - [\Omega(X),\Omega(Y)],
\]
or using exterior product we can write $C = 2(d\Omega -\Omega\wedge\Omega)$. Note 
$\tau_i = \alpha_i\theta^i$ for functions $\alpha_i$, $i=1,2$ using this in the direct computations we obtain
\begin{equation}
\label{curvf}
 C = \begin{pmatrix}
  0 & 0 & 0 & 0 & 0 \\
  0 & 2d\tau_1 & 0 & 0 & -C'  \\
 0 & 0 & 2d\tau_2 & C' & 0  \\
 0 & 0 & 0 & -2d\tau_1 & 0 \\
 0 & 0 & 0 & 0 & -2d\tau_2 
 \end{pmatrix}, 
\end{equation}
where the form $C'$ is given by 
\[
C' = 2(D\omega+\sigma\theta^1\wedge\theta^2).
\]
We denote $\varTheta^{ij}= 2\theta^i\wedge\theta^j$.  

\begin{proposition}
For the Riemann-Christoffel curvature tensor $R(X,Y,Z,W)$ we have
\begin{equation}
\label{curvg}
\begin{array}{rcl}
 R(X,Y,Z,W) & = & g(AX,Z)\;g(AY,W) - g(AY,Z)\; g(AX, W) 
          - 2d\tau_1(X,Y)\; \varTheta^{13}(Z,W) \\ [+6pt]
         & - & 2d\tau_2(X,Y) \; \varTheta^{24}(Z,W)
	+ 2D\omega(X,Y) \; \varTheta^{12}(Z,W).
\end{array}
\end{equation}
 and the forms $\tau_1$, $\tau_2$ and $\omega$ are subject to the 
identities
\begin{equation}
\begin{array}{crclcrcl}
 \label{intcond1}
  &   d\tau_1\wedge \theta^1 &=& 0,  & &  \quad d\tau_2\wedge\theta^2 &=&  0, \\
(a) & \quad d\tau_1\wedge \theta^3 &=&  D\omega \wedge \theta^2,\quad & (b) & \quad d\tau_2 \wedge \theta^4 &=& - D\omega \wedge \theta^1.
\end{array}
 \end{equation}

\end{proposition}
\begin{proof}
On the base of (\ref{curvf}) we find $R(X,Y,V_i,W)$ for each $i=0,\ldots,4$, e.g.
\begin{equation}
\label{rksi}
  \begin{array}{lcl}
  R(X,Y,V_3,W) & = & \theta^1(X) g(AY,W) - \theta^1(Y)g( AX,W) - 2d\tau_1(X,Y)\theta^3(W) \\[+6pt]
             &  & + 2 D\omega(X, Y)\theta^2(W), \\[+6pt]
  R(X,Y,V_4,W) & = & \sigma \theta^2(X)g( AY,W) - \sigma\theta^2(Y) g(AX,W) - 2d\tau_2(X,Y)\theta^4(W)  \\[+6pt]
		& & - 2D\omega(X, Y)\theta^1(W),
 \end{array}
\end{equation}
note that 
\begin{equation}
\label{auxOm2}
 \begin{array}{l}
  \theta^1(X)AY - \theta^1(Y)AX = 2\sigma (\theta_1\wedge\theta_2)(X,Y)V_2, \\[+4pt]
\sigma\theta^2(X)AY - \sigma\theta^2(Y)AX = -2\sigma (\theta_1\wedge\theta_2)(X,Y)V_1,
 \end{array}
\end{equation}

For a vector field $Z = \sum_{k=0}^4 \delta^k(Z)V_k= \sum_{k,j=0}^4 g^{kj}\theta^j(Z) V_k$ hence
\begin{equation*}
 \begin{array}{rcl}
  R(X,Y,Z,W) &=&  \sum\limits_{k=0}^4 \delta^k(Z)R(X,Y,V_k,W) = 
\sum\limits_{k,j=0}^4 g^{kj}\theta^j(Z)R(X,Y,V_k,W)  \\[+12pt]
 &=& \theta^3(Z)R(X,Y,V_1,W) +  \theta^1(Z)R(X,Y,V_3,W) + \theta^4(Z)R(X,Y,V_2,W)  \\ [+6pt]
& & + \; \theta^2(Z)R(X,Y,V_4,W)   \\[+6pt]
&=& -2 d\tau_1(X,Y)\varTheta^{13}(Z,W) - 2d\tau_2(X,Y)\varTheta^{24}(Z,W) + 2D\omega(X,Y)\varTheta^{12}(Z,W) \\[+6pt]
& & +\; (\theta^1(X)\theta^1(Z)+\sigma\theta^2(X)\theta^2(Z))g(AY,W) \\ [+6pt]
& &  - \; (\theta^1(Y)\theta^1(Z) +\sigma\theta^2(Y)\theta^2(Z))g(AX,W), 
 \end{array}
\end{equation*}

and by the Theorem {\bf \ref{Ashape}} (cf. also Remark {\bf \ref{xorient}})
\[
\theta^1(X)\theta^1(Z)+\sigma\theta^2(X)\theta^2(Z) = g(AX,Z) ,
 \]
which being applied in the above formula finishes the proof of (\ref{curvg}).

Applying the first Bianchi identity to (\ref{curvg}) we obtain 
\begin{eqnarray*}
   0  & = & (d\tau_1\wedge \theta^1)\otimes \theta^3 + (d\tau_2\wedge \theta^2)\otimes \theta^4 -
 (d\tau_1\wedge \theta^3 - D\omega\wedge \theta^2)\otimes \theta^1 -
 \\
 & &  (d\tau_2\wedge \theta^4+D\omega\wedge \theta^1) \otimes \theta^2
\end{eqnarray*}
therefore all terms in the brackets must vanish.
\end{proof}

Let $\iota_\xi\kappa$ denotes the inner product of the vector field $\xi$ and an exterior form 
$\kappa$.  It is assumed that 
\begin{equation}
 (\iota_\xi\kappa)(X_1,X_2,\ldots,X_k) = (k+1)\kappa(\xi,X_1,\ldots,X_k),
\end{equation}
where $\kappa$ is a $(k+1)$-form. For an exterior form $\kappa$ we define 
\[
\kappa^0=\kappa - \iota_\xi(\eta\wedge\kappa) =  \eta\wedge \iota_\xi\kappa.
\]
The tensor field $\varphi$ acts on $\kappa$
 $$
:\kappa \mapsto \kappa' = \varphi\kappa: 
 (X_1,\ldots,X_k)\mapsto (\varphi \kappa)(X_1,\ldots,X_k) = \kappa(\varphi X_1,\ldots, \varphi X_k),
$$
note that $\varphi \kappa^0 = \varphi (\eta\wedge\iota_\xi\kappa)= 0$ and the last equation holds by
$\varphi\eta = 0$. A form $\kappa$ is said to be $\varphi$-invariant, resp. $\varphi$-anti-invariant if $\varphi\kappa = \kappa$, resp. $\varphi\kappa = -\kappa$ or $\varphi$-null if $\varphi\kappa = 0$. 
\begin{proposition}
 Each form $\kappa$ has a unique sum decomposition
\begin{equation}
 \kappa = \kappa^+ + \kappa^- + \kappa^0,
\end{equation}
into $\varphi$-invariant $\kappa^+$, $\varphi$-anti-invariant $\kappa^-$ and $\varphi$-null $\kappa^0$ forms 
\begin{equation}
 \varphi\kappa^+ = \kappa^+,\quad \varphi\kappa^- = -\kappa^-, \quad \varphi\kappa^0 = 0.
\end{equation}
\end{proposition}
\begin{proof}
 We set 
\begin{equation}
 \kappa^+ =\dfrac{1}{2}(\kappa+\varphi\kappa - \eta\wedge\iota_\xi\kappa), \quad 
\kappa^- = \dfrac{1}{2}(\kappa - \varphi\kappa -\eta\wedge\iota_\xi\kappa), \quad
\kappa^0 = \eta\wedge\iota_\xi\kappa,
\end{equation}
It is obvious that
\begin{equation}
 \kappa^+ +\kappa^- +\kappa^0 = \kappa.
\end{equation}
Next we directly verify  $\varphi\kappa^+ = \kappa^+$, $\varphi\kappa^- = -\kappa^-$ and $\varphi\kappa^0 = 0$. At first we note that for an arbitrary form 
\[
 \varphi^2\kappa = \kappa-\eta\wedge\iota_\xi\kappa,
\]
then for example
\[
 \varphi\kappa^+ = \dfrac{1}{2}(\varphi\kappa +\varphi^2\kappa) = \kappa^+,
\]
and similarly for $\kappa^-$.
\end{proof}

For a pair of the 2-forms $B$, $D$ by $B\cdot D$ we denote a 4-covariant tensor field $U(X,Y,Z,W)$ defined by
\[
 : (X,Y,Z,W) \mapsto U(X,Y,Z,W) = \dfrac{1}{2} (B(X,Y)D(Z,W)+B(Z,W)D(X,Y)),
\]
in other words $B\cdot D$ is a symmetric product on a space of 2-forms viewed as a space of linear functions on bivectors.

Let $(V_0=\xi, V_1,V_2,V_3,V_4)$ be an arbitrary adopted local frame and $(\theta^0=\eta,\theta^1,\theta^2,\theta^3,\theta^4)$  metric-dual coframe. Let $\varTheta^{ij} = 2\theta^i\wedge\theta^j$.
\begin{proposition}
\label{dtau-Dom-R-Ric}
For the  forms $d\tau_1$, $d\tau_2$ and for the form  $D\omega$, the curvature $R$, the Ricci tensor $Ric$, the scalar curvature $r$ we have the following decompositions
 \begin{equation}
\label{fmdec}
\begin{array}{rcl}
 d\tau_1 & = & \alpha_1\theta^1\wedge\theta^2 + \alpha_2 \theta^1\wedge\theta^3, \\[+4pt]
 d\tau_2 & = &  \beta_1\theta^1\wedge\theta^2 + \beta_2 \theta^2\wedge\theta^4, \\[+4pt]
D\omega & = &  \gamma \theta^1\wedge\theta^2 - \alpha_1\theta^1\wedge\theta^3 - 
\beta_1\theta^2\wedge\theta^4,
\end{array}
\end{equation}
\begin{equation}
\label{curvf2}
\begin{array}{rcl}
  R & =  & (\sigma+2\gamma)\,\varTheta^{12}\cdot\varTheta^{12}- \alpha_2\,\varTheta^{13}\cdot\varTheta^{13} 
        -  \beta_2\, \varTheta^{24}\cdot\varTheta^{24} \\[+4pt]
       & - & 2(\alpha_1\,\varTheta^{12}\cdot\varTheta^{13} +\beta_1\,\varTheta^{12}\cdot\varTheta^{24}),
\end{array}
\end{equation}
\begin{equation}
 \label{ric}
 Ric =   Tr_g ( R) =  
  -2\alpha_2\theta^1\odot\theta^3   -2\beta_2\theta^2\odot\theta^4-2(\alpha_1-\beta_1)\theta^1\odot\theta^2,
\end{equation}
and
\begin{equation}
 \label{scal}
r = -2(\alpha_2+\beta_2),
\end{equation}
functions $\alpha_2$, $\beta_2$ are eigen-functions  each of multiplicity $2$ of the Ricci tensor, thus they are independent of a particular choice of an adopted frame.
\end{proposition}
\begin{proof}
Setting $Z= V_1, V_2$ in the identity
\[
 R(\varphi X, \varphi Y)Z =  - R(X,Y)Z,
\]
we obtain $d\tau_i(\varphi X,\varphi Y) = -d\tau_i(X,Y)$, $i=1,2$. Among the forms $\theta^i\wedge\theta^j$, $i<j$, only 
$\theta^1\wedge\theta^2$, $\theta^1\wedge\theta^3$, $\theta^2\wedge\theta^4$ and $\theta^3\wedge\theta^4$ are $\varphi$-anti-invariant.  Therefore according to the integrability conditions (\ref{intcond1}) we must have
\[
\begin{array}{l}
 d\tau_1 = \alpha_1 \theta^1\wedge\theta^2 + \alpha_2 \theta^1\wedge\theta^3, \\
 d\tau_2 =  \beta_1 \theta^1\wedge\theta^2 + \beta_2 \theta^2\wedge\theta^4,
\end{array}
\]
and in a consequence again by (\ref{intcond1}a,b) 
\[
\begin{cases}
 D\omega & = \theta^2\wedge u -\alpha_1\theta^1\wedge\theta^3, \\
 D\omega & = \theta^1\wedge v - \beta_1\theta^2\wedge\theta^4,
\end{cases} 
\]
for 1-forms $u$, $v$, hence $D\omega =  \gamma \theta^1\wedge\theta^2 - \alpha_1\theta^1\wedge\theta^3-\beta_1\theta^2\wedge\theta^4$ for a function $\gamma$.

Projecting the both sides of the equations (\ref{auxOm2}) onto $W$, next multiplying the first obtained equation by $\theta^1(Z)$ and the second by $\theta^2(Z)$ we find 
\[
 \begin{array}{l}
\theta^1(X)\theta^1(Z)g(AY,W)-\theta^1(Y)\theta^1(Z)g(AX,W) = 2\sigma(\theta^1\wedge\theta^2)(X,Y)\theta^1(Z)\theta^2(W), \\[+4pt]
\sigma\theta^2(X)\theta^2(Z)g(AY,W)-\sigma\theta^(Y)\theta^2(Z)g(AX,W) = -2\sigma(\theta^1\wedge\theta^2)(X,Y)\theta^2(Z)\theta^1(W),
 ×
\end{array}
\]
and summing by sides gives the identity
\begin{equation}
\label{Om2}
 g(AX,Z)g(AY,W)-g(AX,W)g(AY,Z) = \sigma \varTheta^{12}(X,Y) \varTheta^{12}(Z,W).
\end{equation}
Therefore by (\ref{curvg}, \ref{fmdec}, \ref{Om2}) we have (\ref{curvf2}).
 To obtain the Ricci tensor we compute the trace of each component in (\ref{curvf2}) separately
\[
\begin{array}{rcl}
 Ric(Y,Z) &=& (Tr_g R)(Y,Z) = \sum\limits_{m,n=0}^4 g^{mn}R(V_m, Y,Z, V_n) = \\[+6pt]
  & & \sum\limits_{m,n=0}^4g^{mn}U_1(V_m,Y,Z,V_n)+ \ldots + \sum\limits_{m,n=0}^4 g^{mn}U_k(V_m, Y,Z,V_n),
\end{array}
\]
where each $U_s$ is of the form $a\, \varTheta^{ij}\cdot\varTheta^{kl}$ for a function $a$ and some $i,j,k,l$.
Let $U = 2\varTheta^{ij}\cdot\varTheta^{kl}$ and  $S$ be (a symmetric)  tensor field defined by
\[
 : (Y,Z) \mapsto S(Y,Z) = \sum_{m,n=0}^4 g^{mn}U(V_m,Y,Z, V_n).
\]
We directly verify that
\[
 S = 2g^{il}\theta^j\odot\theta^k +2 g^{jk}\theta^i\odot\theta^l 
  -2g^{jl}\theta^i\odot\theta^k-2g^{ik}\theta^j\odot\theta^l
\]
here $\theta^i\odot\theta^k$ is a symmetric product, $2(\theta^i\odot\theta^j)(Y,Z)= \theta^i(Y)\theta^j(Z)+\theta^i(Z)\theta^j(Y)$.  
In the case $i=k$, $j=l$ $ \in \lbrace{1,2,3,4}\rbrace$ the formula simplifies to
\[
 S = 4 g^{ij}\theta^i\odot\theta^j,
\]
as all diagonal coefficients $g^{mm}$ for $m=1,\ldots,4$ vanish $g^{mm}=0$. 

According to (\ref{ric})  symmetric forms 
\[
 Ric +\alpha_2 g, \quad Ric+\beta_2 g,
\]
are both degenerate (of non-maximal rank) hence functions $-\alpha_2$, $-\beta_2$ are eigen-value functions of the Ricci tensor (or operator) thus they are independent of the choice of a local frame.
\end{proof}

\begin{theorem}
\label{riccipot}
If $\mathcal{M}$ is an elliptic or hyperbolic weakly para-cosymplectic five-manifold then the Ricci form $\rho$ is exact and 
\begin{equation*}
 \rho = -2d\tau_1 +2d\tau_2 = 2d(-\tau_1 +\tau_2),
\end{equation*}
\end{theorem}

\begin{proof}
 Note $\theta^i(\varphi X) = g(V_i, \varphi X) =  -\varepsilon_i \theta^i(X)$, here 
$\varphi V_i = \varepsilon_i V_i$ , $\varepsilon_i = \pm 1$, thus 

\begin{equation}
 2(\theta^i\odot\theta^j)(\varphi X,Y) =  -\varepsilon_i \left(\theta^i(X)\theta^j(Y) + \varepsilon_i\varepsilon_j \theta^i(Y)\theta^j(X)\right), 
\end{equation}
now we use (\ref{ric}) $\rho(X,Y) = Ric(\varphi X,Y)$
\[
 \begin{array}{rcl}
\rho(X,Y) &=& -2\alpha_2(\theta^1\odot\theta^3)(\varphi X,Y) - 2\beta_2(\theta^2\odot\theta^4)(\varphi X,Y) 
 -\; 2(\alpha_1-\beta_1)(\theta^1\odot\theta^2)(\varphi X,Y) \\[+4pt]
& = &  -2\alpha_2(\theta^1\wedge\theta^3)(X,Y) + 2\beta_2 (\theta^2\wedge\theta^4)(X,Y) 
- 2(\alpha_1 -\beta_1)(\theta^1\wedge\theta^2)(X,Y)
 ×
\end{array}
\]
and according to (\ref{fmdec}) the right hand side is equal to $-2d\tau_1+2d\tau_2$.
\end{proof}

\subsection{Algebraic classification of the curvature}
The idea is to treat the curvature $R$ purely algebraically. Formally the curvature $R$ is a quadratic form in 
$(\Theta^{12},\Theta^{13},\Theta^{24})$ and fixing this order we associate to the curvature $R$ a matrix $\bar R$ of this quadratic 
form. From (\ref{curvf2}) we find
\begin{equation}
\label{matR}
 \bar R = -\begin{pmatrix}
  -(\sigma+2\gamma)& \alpha_1 & \beta_1 \\
  \alpha_1& \alpha_2  &  0 \\
  \beta_1 & 0 & \beta_2.
 \end{pmatrix}
\end{equation}
 The correspondence $R \leftrightarrow \bar R$ is not 1-1. The local gauge transformation 
\[
(\theta^0,\ldots, \theta^4) \xrightarrow{\alpha} (\theta^{'0},\ldots, \theta^{'4})
\] 
defined by the function $\alpha$ (\ref{gtran}) determines a linear transformation
$P(\alpha): (\Theta^{12}, \Theta^{13}, \Theta^{24})\rightarrow (\Theta^{'12}, \Theta^{'13}, \Theta^{'24})$
\[
 \Theta^{'12} = \Theta^{12}, \quad \Theta^{'13}= \Theta^{13}+\alpha\Theta^{12}, \quad \Theta^{'24}=\Theta^{24}+\alpha\Theta^{12},
\]
and if $\bar R'$ is a matrix of the curvature with respect to $(\Theta^{'12},\Theta^{'13},\Theta^{'24})$ then $\bar R$ and $\bar R'$
 are related by 
\begin{equation}
\label{actionR}
\bar R' = (\alpha) \bar R (\alpha)^T,
\end{equation}
where $(\alpha)$ is a matrix
\[
 (\alpha) =
\begin{pmatrix}
 1 & \alpha & \alpha \\
 0 & 1 & 0 \\
 0 & 0 & 1
\end{pmatrix}.
\]
The formula (\ref{actionR} ) defines a group  $\mathbb{R}$-action on the set of symmetric matrices of the form (\ref{matR}).  If we identify 
these symmetric matrices with points in $\mathbb{R}^5$, then respective action is regular on the (open) set of points which represents matrices with 
non-zero determinant. Orbits are 1-dimensional and corresponding orbifold is a four dimensional manifold.  

A function $I$ (real-or complex-valued) of the  entries of the curvature matrix (\ref{matR})  we call a generalized curvature  invariant 
if this function takes the same values for $\bar R$ and $\bar R'$
\[
 I(\bar R) = I(\bar R'),
\]
 for arbitrary $\alpha$.  Simple examples of invariants are functions of the form $F(\alpha_2,\beta_2)$, where $F$ is a function of two real variables. 
Another obvious invariant is the determinant, $det(\bar R)$.  Here are more interesting examples  
\begin{equation}
\label{i-invariants}
 \begin{array}{rcl}
s(\alpha_2+\beta_2)+(\alpha_1+\beta_1)^2 &=& I_1, \\
\alpha_1\beta_2-\alpha_2\beta_1 &=& I_2.\\
 ×
\end{array} 
\end{equation}

\subsection{The Weyl curvature}
 This part is devoted to obtain an expression for the Weyl curvature tensor similar to (\ref{curvf2}). We begin from simple remarks 
concerning of the Kulkarni-Nomizu product of the two symmetric tensor fields. For arbitrary symmetric tensor field $h(X,Y)$  
\[
 R(X,Y,Z,W)= h(Y,Z)h(X,W)-h(Y,W)h(X,Z),
\]
is algebraic curvature tensor, i.e. $R$ has the the same symmetries as the curvature of Riemannian or pseudo-Riemannian manifold. It can be verified directly or in fact we note that $R$ is exactly of the form of the curvature of a hypersurface in (pseudo)-Euclidean space with the 
second fundamental form $h(X,Y)$. We may treat this formula as a quadratic form in $h$, $R = q(h)$ in consequence we may associate to $q(\cdot)$ its polar form $P(u,v)$ which is a symmetric, bilinear form on symmetric 2-tensors 
$
 P(u,v) = \frac{1}{4×}(q(u+v)-q(u-v)),
$
 and 
\[
 P_{(u,v)}(X,Y,Z,W) = \dfrac{1}{2}(u(Y,Z)v(X,W)-u(Y,W)v(X,Z) + v(Y,Z)u(X,W) - v(Y,W)u(X,Z)),
\]
now the tensor $2P_{(u,v)}$ is called the Kulkarni-Nomizu product of the two symmetric forms $u$, $v$ and usually denoted by $u\wedge v$ which is 
a bit misleading because the product is symmetric in $(u,v)$, $u\wedge v = v\wedge u$. From this construction it is clear that $u\wedge v$ 
is algebraic curvature tensor for arbitrary $(u,v)$. Using Kulkarni-Nomizu product the decomposition of the curvature can be shortly written as
\[
 R = C +\dfrac{1}{n-2}(Ric-\dfrac{r}{2(n-1)}g)\wedge g,
\]
 $g$ denotes a metric, and $Ric$ the Ricci tensor, $r$ a scalar curvature and $n$ is a dimension.  Going back to our considerations 
in the view of the our assumptions and obtained results we have $C = R-\frac{1}{3×}(Ric -\frac{r}{8×}g)\wedge g$, now from (\ref{ric}, \ref{scal}) and $g=\theta^0\odot\theta^0+2\theta^1\odot\theta^3+2\theta^2\odot\theta^4$, 
\[
 Ric -\dfrac{r}{8}g = \frac{(\alpha_2+\beta_2)}{4×}\theta^0\odot\theta^0 +\frac{(\beta_2-3\alpha_2)}{2×}\theta^1\odot\theta^3+\frac{(\alpha_2-3\beta_2)}{2×}\theta^2\odot\theta^4 -2(\alpha_1-\beta_1)\theta^1\odot\theta^2.
\]
Directly from the definition of the Kulkarni-Nomizu product we verify that 
\[
 (2\theta^i\odot\theta^j)\wedge(2\theta^k\odot\theta^l) = -2\Theta^{ik}\cdot\Theta^{jl}-2\Theta^{il}\cdot\Theta^{jk},
\]
using bilinearity of the Kulkarni-Nomizu product we obtain $(Ric-\frac{r}{8×}g)\wedge g$ in the form 
\begin{equation}
\label{expand}
\begin{array}{rcl}
(Ric-\frac{r}{8×}g)\wedge g &=& (\sum\limits_{i\leqslant j} 2s_{ij}\theta^i\odot\theta^j)\wedge(\sum\limits_{k\leqslant l} 2c_{kl}\theta^k\odot\theta^l) = \sum\limits_{i\leqslant j,k\leqslant l} s_{ij}c_{kl} (2\theta^i\odot\theta^j)\wedge(2\theta^k\odot\theta^l) = \\[+12pt]
&=& \sum\limits_{i\leqslant j,k\leqslant l} -2s_{ij}c_{kl}(\Theta^{ik}\cdot\Theta^{jl}+\Theta^{il}\cdot\Theta^{jk}).
 ×
\end{array} 
\end{equation}
We represent a quadratic form $q=\sum_{i}q_{i}v_i^2 + \sum_{i<j}2q_{ij}v_iv_j$ as non-oriented graph with vertices as variables and the two vertices 
$v_i$, $v_j$ are joined if term $2q_{ij}v_iv_j$ appears in the formula for $q$ ($q_{ij}\neq 0$) and a loop 
matches a quadratic term $q_{i}v_i^2$. Edges of the graph are labelled by the coefficients $q_{ij}$. For example $2v_1^2- 3v_1v_2+v_2v_3-4v_3^2$ is represented by       
$$
\xymatrix{
 v_1 \ar@(ul,dl)@{-}[]^{2} \ar@{-}[rr]^{-\frac{3}{2×}}
 && v_2 \ar@{-}[rr]^{\frac{1}{2×}}
 && v_3 \ar@(ur,dr)@{-}[]^{-4} }
$$
The Weyl curvature $C$ in general does not share the $\varphi$-invariance properties of the curvature.
\begin{proposition} 
\label{WCdec}
The  quadratic form $q(C) = C_1+C_2+C_3$ in $(\Theta^{ij})$, $i<j$, $i=0,\ldots,4$, $j=1,\ldots 4$ by (\ref{curvf2}, \ref{expand}) which represents the Weyl curvature has the following graph representation
$$
C_1: \xymatrix{
 \Theta^{01} \ar@{-}[rr]^{\frac{\beta_1-\alpha_1}{3}} && \Theta^{02} \\
  \Theta^{03} \ar@{-}[u]^{\frac{\beta_2-\alpha_2}{6}}  && \Theta^{04} \ar@{-}[u]^{\frac{\alpha_2-\beta_2}{6×}} } 
$$
\vspace{.5cm}
$$
C_2: \xymatrix{
\Theta^{24} \ar@(ul,dl)@{-}_{-\frac{\alpha_2+3\beta_2}{6}} \ar@{-}[rr]^{-\frac{\alpha_1+2\beta_1}{3}} 
 && \Theta^{12} \ar@(ur,ul)@{-}_{\sigma+2\gamma} \ar@{-}[rr]^{-\frac{2\alpha_1+\beta_1}{3}} 
 && \Theta^{13} \ar@(ur,dr)@{-}^<{-\frac{3\alpha_2+\beta_2}{6}}   \\
&& \Theta^{34} \ar@{-}[u]^{r/12} }
$$
\vspace{.5cm}
$$
C_3: \xymatrix{ 
 \Theta^{14} \ar@{-}[rr]^{-r/12} && \Theta^{23} }
$$
\end{proposition}
As a simple application of the above formula we have the following result
\begin{proposition}
 If Weyl curvature operator $C$ commutes with the structure affinor $\varphi$ 
\[
 C(X,Y)\varphi Z = \varphi C(X,Y)Z,
\]
then manifold is Ricci-flat and in consequence  $R = C$.
\end{proposition}
\begin{proof}
 The summands $C_1$, $C_2$ in the Proposition \ref{WCdec} must vanish identically which yields   
that the Ricci tensor  is $\eta$-Einstein and the scalar curvature also vanishes $r=0$.  Thus $Ric=0$, the 
manifold is Ricci flat and $C=R$.  
\end{proof}

\section{$\eta$-Einstein manifolds}
\label{secEta}
This section is devoted to  the study of a particular class of manifolds for which the Ricci tensor satisfies
\begin{equation}
Ric(X,Y) = \dfrac{r_0}{4} (g(X,Y) - \eta(X)\eta(Y)), 
\end{equation}
we denote the scalar curvature by $r_0$ to emphasize that it must be constant according to the Theorem \ref{scEtaEin}. Let $(\mathcal{M},\varphi,\xi,\eta,g)$ be a elliptic or hyperbolic weakly para-cosymplectic manifold with 
para-K\"ahler leaves. Let fix a local adopted frame $(V_0=\xi, V_1,\ldots, V_4)$. Assuming $Ric$ is $\eta$-Einstein
by (\ref{ric})
\begin{eqnarray*}
Ric &=& \dfrac{r_0}{4}(g-\eta\odot\eta) = \dfrac{r_0}{2}(\theta^1\odot\theta^3+\theta^2\odot\theta^4) \\[+4pt]
   & = &  -2\alpha_2\theta^1\odot\theta^3-2\beta_2\theta^2\odot\theta^4-2(\alpha_1-\beta_1)\theta^1\odot\theta^2, 
\end{eqnarray*}
hence $\alpha_2=\beta_2 = -r_0/4$ and $\alpha_1 = \beta_1= a$ and by
(\ref{fmdec}) the differentials $d\tau_1$, $d\tau_2$ of the  forms satisfy
\begin{equation*}
 \begin{array}[]{rcl}
d\tau_1 &=& a\theta^1\wedge\theta^2 -\dfrac{r_0}{4}\theta^1\wedge\theta^3, \\[+8pt]
d\tau_2 &=& a\theta^1\wedge\theta^2 -\dfrac{r_0}{4}\theta^2\wedge\theta^4,
 ×
\end{array}
\end{equation*}
now the idea is to note that the coefficients $\alpha_1$, $\beta_1$ are dependent on the choice of an adopted
frame and that for suitable chosen such adopted frame we can assure $a=0$ if $r_0\neq 0$. Indeed if 
$(V_0'=\xi,V_1',\ldots,V'_4)$ is another adopted frame (\ref{gtran}), then as the  forms $\tau_i$ are gauge invariant we can write
\begin{equation*}
 \begin{array}[]{rcl}
d\tau_1 &=& a'\theta^{1'}\wedge\theta^{2'} -\dfrac{r_0}{4}\theta^{1'}\wedge\theta^{3'}, \\[+8pt]
d\tau_2 &=& a'\theta^{1'}\wedge\theta^{2'} -\dfrac{r_0}{4}\theta^{2'}\wedge\theta^{4'},
 ×
\end{array}
\end{equation*}     
here the relations between the forms $\theta^i$ and $\theta^{i'}$ are 
\[
 \theta^{3'} = \theta^3+\alpha\theta^2,\quad \theta^{4'}= \theta^4-\alpha\theta^1,
\]
 for the gauge function $\alpha$. Thus simple comparison follows
\[
 a = a' - \dfrac{r_0}{4}\alpha,
\]
 so if scalar curvature is non-zero setting $\alpha = - 4a/r_0$ enforce $a'$ to vanish.  
\begin{proposition}
 For $\eta$-Einstein manifold with non-zero (constant) scalar curvature near each point there is a unique up
to change of signs adopted frame such that
\begin{equation}
\label{streq1}
 \begin{array}{rcl}
d\tau_1 &=& -\dfrac{r_0}{4}\theta^1\wedge\theta^3, \\[+8pt]
d\tau_2 &=& -\dfrac{r_0}{4}\theta^2\wedge\theta^4.
 ×
\end{array}
\end{equation}
In the case of the Ricci-flat manifold when the scalar curvature $= 0$ we have
\begin{equation}
 \label{streq2}
d\tau_1 = d\tau_2 = a\theta^1\wedge\theta^2,
\end{equation}
in an arbitrary adopted frame.
\end{proposition}

\subsection{The local classification of scalar non-flat $\eta$-Einstein manifolds}
From (\ref{nabvi},\ref{OmegaMat}) we obtain the first set of Cartan's structure equations for 
an adopted metric-dual coframe 
\begin{equation}
\label{cartans}
\begin{array}{rcl}
d\theta^0 &=& d\eta = -\theta^1\wedge\theta^1-\sigma\theta^2\wedge\theta^2 = 0, \\[+4pt]
d\theta^1 &=& \tau_1\wedge\theta^1 = 0, \\[+4pt]
d\theta^2 &=& \tau_2\wedge\theta^2 = 0, \\[+4pt]
d\theta^3 &=& \theta^1\wedge\theta^0+\omega\wedge\theta^2 -\tau_1\wedge\theta^3, \\[+4pt]
d\theta^4 &=& \sigma\theta^2\wedge\theta^0 -\omega\wedge\theta^1 -\tau_2\wedge\theta^4.
 ×
\end{array}
\end{equation}
To this set of equations we have to add (\ref{streq1}) and their consequences: taking the exterior 
derivative of the both hands having in mind $d^2=0$ and $d\theta^0=d\theta^1=d\theta^2=0$ we obtain
\begin{equation}
 \label{cstreq1}
\theta^1\wedge d\theta^3 = 0,\quad \theta^2\wedge d\theta^4 = 0.
\end{equation}

\begin{theorem}
\label{th-eta-einst}
 Let $\mathcal{U}\subset \mathbb{R}^5 = \mathbb{R}\times\mathbb{R}^4$ be a domain and $(t,x^1,x^2,y^1,y^2)$ be
global Cartesian coordinates on $\mathbb{R}^5$. We define a coframe of local 1-forms on $\mathcal{U}$ 
\begin{equation}
\label{locdes}
 \begin{array}{l}
\theta^0 = dt, \\[+4pt]
\theta^1 = dx^1, \\[+4pt]
\theta^2 = dx^2, \\[+4pt]
\theta^3 = df+udx^1, \quad f = f^0(x^1)-y^1,\quad u = u^0(x^1,x^2)-t+\dfrac{r_0}{8}(y^1)^2, \quad r_0 \neq 0 \\[+8pt]
\theta^4 = dh +vdx^2, \quad h = h^0(x^2)+y^2, \quad v = v^0(x^1,x^2) -\sigma t+\dfrac{r_0}{8}(y^2)^2,
 ×
\end{array}
\end{equation}
for arbitrary smooth functions $f^0$, $u^0$, $h^0$, $v^0$ on $\mathcal{U}$ and a non-zero constant $r_0$, 
$\sigma= \pm 1$, and we set an almost para-contact metric structure $(\varphi,\xi, \eta, g)$ on $\mathcal{U}$ 
\begin{equation}
\label{locstruct}
 \begin{array}{l}
\xi = \dfrac{\partial}{\partial t}, \quad \eta = \theta^0 = dt, \\[+6pt]
\varphi \theta^0 = 0,\quad \varphi \theta^1 = \theta^1,\quad \varphi \theta^2 = -\theta^2, 
\quad \varphi\theta^3 = \theta^3,\quad \varphi\theta^4 = -\theta^4, \\[+4pt]
g =  dt\odot dt + 2 \theta^1\odot\theta^3+2 \theta^2\odot\theta^4,
 ×
\end{array}
\end{equation}
then $\mathcal{U}$ equipped with this structure is an almost para-cosymplectic manifold, weakly para-cosym\-plectic,
with para-K\"ahler leaves and $\eta$-Einstein with constant scalar curvature $r_0$. For $\sigma=1$ we have elliptic manifold, for $\sigma=-1$ hyperbolic. Conversely, if 
$\mathcal{M}$ is weakly para-cosymplectic manifold with para-K\"ahler leaves with $\eta$-Einstein 
Ricci tensor then near a point $p\in \mathcal{M}$ there is an adopted metric-dual coframe $(\theta^0=\eta, \theta^1,\ldots,\theta^4)$ and local coordinates $(t,x^1,x^2,y^1,y^2)$ near $p$ such that (\ref{locdes}) are
fulfilled. 
\end{theorem}
\begin{proof}
 \noindent ($\Rightarrow$) Essentially here we use the uniqueness of the Levi-Civita connection. Assuming 
(\ref{locdes}) we find that (\ref{cartans}) are satisfied with 
\[
 \tau_1 = -\dfrac{r_0}{4}y^1dx^1, \quad \tau_2 = \dfrac{r_0}{4×}y^2dx^2, \quad 
\omega = -\dfrac{\partial u^0}{\partial x^2}dx^1+\dfrac{\partial v^0}{\partial x^1}dx^2,
\]
the connection form for the metric (\ref{locstruct}) is, by the uniqueness, given as in (\ref{OmegaMat}) so 
all the curvature expressions are valid, especially (\ref{fmdec}) and we verify this directly (\ref{streq1}) 
so the Ricci tensor is $\eta$-Einstein and the constant $r_0$ is exactly the scalar curvature.

\noindent ($\Leftarrow$) Our considerations are purely local - we may shrink if necessary a neighborhood of a some fixed point 
$p\in\mathcal{M}$ such that $\theta^1=dx^1$, $\theta^2 = dx^2$ for some locally defined functions
$x^1$, $x^2$. Moreover we can assume $d\theta^3 = u_1\wedge\theta^1 = u_1\wedge dx^1$ and 
$d\theta^4 = u_2\wedge\theta^2 = u_2\wedge dx^2$ for a local 1-forms $u_1$, $u_2$. Therefore 
the pullback $\bar\theta^3$ on a local hypersurface $x^1 = const$ is a closed form thus
\[
 \bar\theta^3 = d\bar f_{x^1},
\]
 $\bar f_{x^1}$ is a 1-parameter family of smooth functions, $x^1$ playing a role of the parameter. We
put $f(x^1,\cdot) = \bar f_{x^1}(\cdot)$ which properly defines a smooth local function on a some neighborhood
of the point $p$. The form $\theta^3$ can be written locally as
\[
 \theta^3 = df+ u dx^1 = df+ u\theta^1,
\]
for a function $u$. If we replace $\theta^1$ by $\theta^2$ and $\theta^3$ by $\theta^4$ we obtain a similar decomposition for 
the form $\theta^4$
\[
 \theta^4 = dh +v dx^2 = dh + v\theta^2.
\]
From (\ref{cartans},\ref{cstreq1}) it follows that 
\[
 \omega\wedge\theta^1\wedge\theta^2 = 0,
\]
in consequence $\omega= a\theta^1 + b\theta^2$ for functions $a$, $b$. Moreover due to (\ref{fmdec}) and our general 
assumption that the manifold is $\eta$-Einstein we have 
\[
 D\omega = d\omega-\omega\wedge(\tau_1+\tau_2) = \gamma \theta^1\wedge\theta^2,
\]
thus both $a$, $b$ must satisfy
\[
 da = a_1\theta^1+a_2\theta^2,\quad db = b_1\theta^1+b_2\theta^2,
\]
and we find 
\[
 \gamma = b_1-a_2 +a\alpha_2-b\alpha_1,
\]
where $\tau_1 = \alpha_1\theta^1$, $\tau_2=\alpha_2\theta^2$. We need a suitable coordinate system near our fixed 
point $p$. The obvious choice is to try to extend local functions $(t,x^1,x^2)$, where 
$dt = \eta=\theta^0$, $dx^1= \theta^1$, $dx^2=\theta^2$. Note 
\[
 \eta\wedge d\tau_1\wedge d\tau_2 = \dfrac{r_0^2}{16}\eta\wedge\theta^1\wedge\theta^3\wedge\theta^2\wedge\theta^4 \neq 0,
\]
everywhere from the other hand
\[
 \eta\wedge d\tau_1 \wedge d\tau_2 = \eta \wedge d\alpha_1 \wedge \theta^1\wedge d\alpha_2 \wedge\theta^2, 
\]
so the 1-forms $(\eta=dt, \theta^1=dx^1, \theta^2 = dx^2, d\alpha_1, d\alpha_2)$ are closed and pointwise linearly 
independent which suggest to treat the functions $\alpha_1$, $\alpha_2$ as additional independent variables. 
However we will use coordinates $y^1$, $y^2$ defined by $\alpha_1= -\frac{r_0}{4}y^1$, $\alpha_2= \frac{r_0}{4×}y^2$
so a 1-form $\pi=\eta-2\tau_1+2\tau_2$ which is a globally defined contact form on $\mathcal{M}$ in coordinates
$(t,x^1,x^2,y^1,y^2)$ is expressed by 
\[
 dt +\frac{r_0}{2×}y^1 dx^1 + \frac{r_0}{2}y^2 dx^2.
\]
The local functions $f$, $h$, $u$, $v$ have to be established on the base of the Cartan's equations (\ref{cartans}). In coordinates 
\[
\begin{array}{l}
du\wedge dx^1 = dx^1\wedge dt +a dx^1\wedge dx^2 +\dfrac{r_0}{4}y^1dx^1\wedge df, \\[+8pt]
dv \wedge dx^2 = \sigma dx^2 \wedge dt -b dx^2\wedge dx^1 - \dfrac{r_0}{4}y^2 dx^2\wedge dh, 
 \end{array}
\]                   
and possible solution of this system of equations are of the form as in (\ref{locdes}).
\end{proof}

\section{Manifolds with contact Ricci potential}
\label{secCont}
As we know the Ricci form $\rho$  of a hyperbolic or elliptic  manifold is closed and exact and $\rho =2d(\eta-\tau_1+\tau_2)$, (\ref{riccipot}). The 
1-form $\eta-\tau_1+\tau_2$ we call a Ricci potential.  In this section we shall describe locally manifolds for which the Ricci potential is a contact 
form on a 5-dimensional manifold, equivalently manifolds with $\eta\wedge\rho\wedge\rho\neq 0$ everywhere.  Obviously Ricci potential is not determined 
uniquely; arbitrary $-\tau_1+\tau_2+ \beta$ with $d\beta=0$ will do but our choice is dictated by an attempt to construct in natural way a contact form.  
If $-\tau_1+\tau_2+\beta$ is another contact Ricci potential then $\beta = u\eta+\beta_0$ for a nonvanishing function $u$. Taking in non-canonical way 
the Ricci potential, that is $\beta \neq \eta$ influences the contact distribution $\mathcal{C}_\beta: -\tau_1+\tau_2+\beta=0$. So, strictly speaking 
we should say about family of contact structures, also having in mind different possible choices of contact distributions. Our choice $\beta=\eta$ although 
it seems to be natural in other settings may become inconvenient. 

\begin{theorem}
\label{th-cont-ric-pot}
 Let $\mathcal{U}\subset \mathbb{R}^5 = \mathbb{R}\times \mathbb{R}^4$ be a domain and $(t,x^1,x^2,y^1,y^2)$ be global 
Cartesian coordinates on $\mathbb{R}^5$. Let define a coframe of 1-forms on $\mathcal{U}$
\begin{equation}
\label{locforms3}
 \begin{array}{l}
\theta^0 = dt, \\
\theta^1 = dx^1, \\
\theta^2 = dx^2, \\
\theta^3 = \dfrac{\partial f_1}{\partial y^1}dy^1 + (u-t+\dfrac{\partial f_1}{\partial x^1})dx^1 + 
(f_3+\dfrac{\partial f_1}{\partial x^2})dx^2, \\[+10pt]
\theta^4 = \dfrac{\partial f_2}{\partial y^2}dy^2 + (-f_3+\dfrac{\partial f_2}{\partial x^1})dx^1 +
(v-\sigma t+\dfrac{\partial f_2}{\partial x^2})dx^2,
 ×
\end{array}
\end{equation}
where the functions $(f_1,u)$, $(f_2,v)$ satisfy the following conditions 
\begin{equation}
\label{constraints}
 \begin{array}{l}
f_1=f_1(x^1,x^2,y^1), \quad \dfrac{\partial f_1}{\partial y^1} \neq 0,  \quad
\dfrac{\partial u}{\partial y^1} +y^1\dfrac{\partial f_1}{\partial y^1}= 0,  \\[+10pt]
f_2 =f_2(x^1,x^2,y^2),\quad \dfrac{\partial f_2}{\partial y^2} \neq 0, \quad
 \dfrac{\partial v}{\partial y^2}-y^2\dfrac{\partial f_2}{\partial y^2}=0, \quad \sigma =\pm 1,
 ×
\end{array}
\end{equation}
and $f_3$ is arbitrary function on $\mathcal{U}$.
Let define an almost para-contact metric structure $(\varphi,\xi,\eta,g)$ on $\mathcal{U}$
\begin{equation}
\label{struct4}
 \begin{array}{l}
\xi = \dfrac{\partial}{\partial t}, \quad \eta= \theta^0 = dt, \\[+8pt]
\varphi \theta^0 = 0, \quad \varphi\theta^1 = \theta^1,\quad \varphi\theta^2 = -\theta^2, \quad
\varphi\theta^3 = \theta^3, \quad \varphi\theta^4 = -\theta^4,\\[+4pt]
g = dt\odot dt + 2\theta^1\odot\theta^2 +2 \theta^2\odot\theta^4, 
 ×
\end{array}
\end{equation}
then $\mathcal{U}$ equipped with this structure is a weakly para-cosymplectic manifold with para-K\"ahler leaves,
hyperbolic for $\sigma=-1$, elliptic for $\sigma =+1$, and 
with a contact Ricci potential 
\begin{equation}
\label{localcontact}
 \eta -\tau_1+\tau_2 = dt +y^1dx^1+y^2dx^2.
\end{equation}
The eigen-functions $\alpha$, $\beta$ of the Ricci tensor $Ric$ are given by 
\begin{equation}
 \alpha= \left(\dfrac{\partial f_1}{\partial y^1}\right)^{-1}, \quad \beta = -\left(\dfrac{\partial f_2}{\partial y^2}\right)^{-1}.
\end{equation}
Conversely if $\mathcal{M}$ is weakly para-cosymplectic with para-K\"ahler leaves, elliptic or hyperbolic, with 
contact Ricci potential $\eta-\tau_1+\tau_2$, then near arbitrary point $p\in \mathcal{M}$ we can find a local 
coordinates system $(t,x^1,x^2,y^1,y^2)$ such that (\ref{locforms3},\ref{constraints},\ref{struct4},\ref{localcontact}) are satisfied.  
\end{theorem}
\begin{proof}
 \noindent ($\Rightarrow$) Proof here goes as in   the proof of the Theorem \ref{th-eta-einst}. We directly verify that 
(\ref{locforms3}, \ref{constraints}, \ref{struct4}) define an almost para-cosymplectic structure on $\mathcal{U}$ satisfying the 
Cartan's structure equations (\ref{cartans}) with the 1-forms $\tau_1$, $\tau_2$, $\omega$ given by 
\[
 \tau_1 = -y^1dx^1, \quad \tau_2 = y^2dx^2, 
\]
 $\omega$ can be described in the following way
\begin{equation}
\label{omega}
 \omega = df_3 +u_3dx^1+v_3dx^2= df_3+u_3\theta^1+v_3\theta^2,
\end{equation}
and the functions $u_3$, $v_3$ are defined by the identities
\begin{equation}
\label{u3v3}
  \begin{array}{r}
-\dfrac{\partial u}{\partial x^2}-y^1(\dfrac{\partial f_1}{\partial x^2×} +f_3) =u_3, \\[+10pt]
\dfrac{\partial v}{\partial x^1×} -y^2(\dfrac{\partial f_2}{\partial x^1×} -f_3) = v_3,
 ×
\end{array}
\end{equation}

\noindent ($\Leftarrow$) The main line of local classification of such manifolds lies in the fact that we are able 
to choose in convenient way a local coordinates system $(t,x^1,x^2,y^1,y^2)$ similarly as it was done in the classification of $\eta$-Einstein manifolds
\begin{equation}
\label{matchcoord}
 \eta=\theta^0=dt, \quad \theta^1=dx^1, \quad \theta^2 = dx^2 , \quad \tau_1 = -y^1dx^1, \quad \tau_2 = y^2 dx^2. 
\end{equation}
 The condition 
of the existence of the contact Ricci potential can be written in several equivalent forms 
\[
 \eta\wedge\rho\wedge\rho \neq 0 \quad \Leftrightarrow \quad \eta\wedge d\tau_1\wedge d\tau_2\neq 0 \quad\Leftrightarrow \quad r= -2(\alpha+\beta),\;\; 
\alpha\beta\neq 0,
\]
 the last condition simply means that the eigenvalues of the Ricci tensor are non-zero at each point, except obviously the eigenvalue 
which corresponds to $\xi$, as Ricci tensor is always degenerate and $Ric(\xi,\cdot)=0$. In the other words the Ricci form has contact 
potential if and only if the kernel of the Ricci tensor is 1-dimensional, and therefore is spanned by $\xi$.  The presence of the 
coordinates  (\ref{matchcoord}) allows us in efficient way to analyze the Cartan structure equations (\ref{cartans}). 
At first we note  as consequences of these  equations that
\begin{equation}
 \label{wedges}
d\theta^3\wedge\theta^1\wedge\theta^2= d\theta^4\wedge\theta^1\wedge\theta^2 = d\omega\wedge\theta^1\wedge\theta^2 = 0,
\end{equation}
thus locally forms $(\theta^3, \theta^4,\omega)$ are given by
\begin{equation}
 \label{locforms}
\begin{array}{l}
 \theta^3 = df_1' + u_1' dx^1 +v_1' dx^2, \\
 \theta^4 = df_2' + u_2' dx^1 +v_2' dx^2,\\
 \omega = df_3' + u_3' dx^1 + v_3' dx^2
 ×
\end{array}
\end{equation}
where $(f_i')=f_i'(t,x^1,x^2,y^1,y^2))$, $(u_i')=(u_i'(t,x^1,x^2,y^1,y^2))$, $(v_i')=(v_i'(t,x^1,x^2,y^1,y^2))$ are locally defined 
functions which have to satisfy some additional integrability conditions according to (\ref{cartans}). Plugging (\ref{locforms})
into (\ref{cartans}) and multiplying exteriorly obtained equations by $dx^1$ and $dx^2$ we get the following conditions 
\begin{equation}
\label{intcond3}
 \begin{array}{rcl}
(du_1'+dt +y^1df_1')\wedge dx^1\wedge dx^2 &=& 0, \\
(dv_1'-df_3')\wedge dx^1 \wedge dx^2 &=& 0, \\
(du_2' +df_3')\wedge dx^1 \wedge dx^2 &=& 0, \\
(dv_2'+\sigma dt-y^2df_2') \wedge dx^1 \wedge dx^2 &=& 0.
 ×
\end{array} 
\end{equation}

In what follow given a function $F$ by $F_{y^i}$, $F_{x^i}$, $i=1,2$ we denote its first partial derivatives with respect to coordinate 
functions $(x^1,x^2,y^1,y^2)$, so for example $f'_{1y^1} = \partial f'_1/\partial y^1$.
The exterior differential of the first and fourth equations follow 
\[
 df_1'\wedge dy^1\wedge dx^1\wedge dx^2=0, \quad df_2' \wedge dy^2\wedge dx^1 \wedge dx^2 = 0,
\]
thus $f_1'=f_1'(y^1,x^1,x^2)$ and $f_2'= f_2'(y^2,x^1,x^2)$ in consequence 
\[
 u_1' = u_1^{'0}(y^1,x^1,x^2) - t, \quad v_2' = v_2^{'0}(y^2,x^1,x^2)-\sigma t,
\]
and functions $u_1^{'0}$, $v_2^{'0}$ have to satisfy 
\[
 u_{1y^1}^{'0} +y^1 f'_{1y^1}=0, \quad
v^{'0}_{2y^2} -y^2f'_{2y^2} =0.
\]
From the second and the third equations of (\ref{intcond3}) we find 
\[
 v_1' = f_3' + v_1^{'0}(x^1,x^2), \quad u_2' = -f_3+u_2^{'0}(x^1,x^2). 
\]
The function $f_3'$ is arbitrary however for the functions $u_3'$, $v_3'$ again from (\ref{cartans}) we obtain the following 
relations 
\[
 \begin{array}{l}
-u'_{1x^2}+\dfrac{\partial (v_1'-f_3')}{\partial x^1×} 
-y^1(f'_{1x^2×} +v'_1) =u_3', \\[+10pt]
-\dfrac{\partial (u_2'+f_3')}{\partial x^2×} +v'_{2x^1×} 
-y^2(f'_{2x^1×} + u_2') =v_3',
 ×
\end{array}
\]
we treat these relations as definitions for $u_3'$ and $v_3'$.  Now (\ref{locforms}) can be written in the following form
\begin{equation}
 \label{locforms2}
\begin{array}{l}
\theta^3 = f'_{1y^1}dy^1 +(u_1^{'0}-t+f'_{1x^1})dx^1 + 
(f_3'+v_1^{'0}+f'_{1 x^2})dx^2, \\[+10pt]
\theta^4 = f'_{2y^2}dy^2 + (-f_3'+u_2^{'0}+f'_{2x^1})dx^1 + 
(v_2^{'0}-\sigma t + f'_{2x^2})dx^2,
 ×
\end{array}
\end{equation}
with obvious conditions $f'_{1y^1×}\neq 0$, $ f'_{2y^2×} \neq 0$ everywhere. 

Let $f_1$, $f_2$, $u$, $v$ 
are functions defined by 
\[
 \begin{array}{l}
\begin{cases}
 f_{1x^2} & = v_1^{'0}+f'_{1 x^2},\\[+6pt]
 f_{1y^1} & = f'_{1y^1},
\end{cases}
\qquad
\begin{cases}
f_{2x^1} & = u_2^{'0}+f'_{2x^1},\\[+6pt]
f_{2y^2} & = f'_{2y^2}, 
\end{cases}
\quad
\\[+24pt]
u = u_1^{'0}+f'_{1x^1}-f_{1x^1},\quad
v = v_2^{'0}+ f_{2x^2}-f_{2x^2},
 ×
\end{array}
\]
then the forms in (\ref{locforms2}) can be written in simpler way, $f_3 = f_3'$
\[
\begin{array}{l}
\theta^3 = f_{1y^1}dy^1+ (u-t+f_{1x^1})dx^1 
+ (f_3 + f_{1x^2})dx^2,\\[+8pt]
\theta^4 = f_{2y^2}dy^2+ (-f_3+f_{2x^1})dx^1 + 
(v-\sigma t +f_{2x^2})dx^2,
 ×
\end{array}
\]
and consequently for the functions $u_3 = u_3'$, $v_3 = v_3'$ in (\ref{u3v3}) we find
\[
\begin{array}{r}
-u_{x^2}-y^1(f_3+f_{1x^2}) = u_3,\\[+6pt]
  v_{x^1} -y^2(-f_3+f_{2x^1}) = v_3,
 ×
\end{array}
\]

The functions $f_{1y^1}$, $f_{2y^2}$ are directly related to eigenvalues functions $\alpha_2$, $\beta_2$ of the Ricci tensor as
\[
 d\tau_1 = -dy^1\wedge dx^1 = -dy^1\wedge \theta^1, \quad d\tau_2 = dy^2\wedge dx^2 = dy^2\wedge \theta^2, 
\]
expressing $dy^1$ and $dy^2$ through the forms $(\eta=\theta^0, \theta^1,\ldots,\theta^4)$ we find 
\begin{equation}
\label{dy1dy2}
 \begin{array}{l}
dy^1 = \dfrac{1}{f_{1y^1}×}\theta^3 
-(\dfrac{u-t+f_{1x^1}}{f_{1y^1}×})\theta^1 
- (\dfrac{f_3+f_{1x^2}}{f_{1y^1}×})\theta^2, \\[+12pt]
dy^2 = \dfrac{1}{f_{2y^2}×}\theta^4 
 -(\dfrac{-f_3+f_{2x^1}}{f_{2y^2}×})\theta^1 
- (\dfrac{v-\sigma t+f_{2x^2}}{f_{2y^2}×})\theta^2,
 ×
\end{array}
\end{equation}
hence 
\begin{equation}
\label{dt1dt2-100}
\begin{array}{l}
d\tau_1 = -(\dfrac{f_3+f_{1x^2}}{f_{1y^1}×})\theta^1\wedge \theta^2 
           + \dfrac{1}{f_{1y^1}×}\theta^1\wedge\theta^3, \\[+10pt]
d\tau_2 = -(\dfrac{-f_3+f_{2x^1}}{f_{2y^2}×})\theta^1\wedge \theta^2 
          -\dfrac{1}{f_{2y^2}×}\theta^2 \wedge \theta^4,
\end{array}
\end{equation}
which being compared to (\ref{fmdec}) follows 
\[
 \alpha_2 = \dfrac{1}{f_{1y^1}},\quad \beta_2 = -\dfrac{1}{f_{2y^2}},
\]
and for the Ricci tensor $Ric$ and scalar curvature $r$ we obtain (\ref{fmdec},\ref{ric},\ref{scal}) 
\[
\begin{array}{rcl}
Ric &=& -\dfrac{2}{f_{1y^1}}\theta^1\odot \theta^3 + \dfrac{2}{f_{2y^2}}\theta^2\odot\theta^4 
- 2(\alpha_1-\beta_1)\theta^1\odot\theta^2, \\[+10pt]
&& \alpha_1-\beta_1 = \left(\dfrac{f_{2x^1}}{f_{2y^2}} -\dfrac{f_{1x^2×}}{f_{1y^1}}\right)
- f_3\left(\dfrac{f_{1y^1}+f_{2y^2}}{f_{1y^1}f_{2y^2}}\right)×, \\[+10pt]
r &=& 2\dfrac{f_{1y^1}-f_{2y^2}}{f_{1y^1}f_{2y^2}},
 ×
\end{array}
\]
\end{proof}

\subsection{Generalized $\eta$-Einstein manifolds}
In the decomposition of the Ricci tensor (\ref{ric}) the part 
\[
 Ric_0 = -2(\alpha_1-\beta_1)\theta^1\odot\theta^2,
\]
is geometrically significant only in the points where the eigenfunctions $\alpha_2$, $\beta_2$ of the Ricci tensor are equal 
$\alpha_2= \beta_2 = -r/4$. If at a point $p\in \mathcal{M}$ $\alpha_2(p) \neq \beta_2(p)$ then in a small neighborhood 
$\mathcal{U}_p$, $\alpha_2 \neq \beta_2$ and we can choose a new adopted coframe $(\theta^0=\eta, \theta^{1'},\ldots,\theta^{4'})$
on $\mathcal{U}_p$ by appropriate gauge transform, such that in this new frame $Ric_0' = 0$ everywhere on $\mathcal{U}_p$. 

The manifold $\mathcal{M}$ is said to be generalized $\eta$-Einstein if the Ricci tensor satisfies the following condition
\begin{equation}
 Ric(X,Y) = \dfrac{r}{4}(g(X,Y)-\eta(X)\eta(Y)) + Ric_0(X,Y),
\end{equation}
where $r$ is scalar curvature, $r \neq 0$ everywhere, and  
\[
   Ric_0 = -2(\alpha_1-\beta_1) \theta^1\odot\theta^2.
\]
The possible case $r=0$ identically we treat as degenerate because such manifolds cannot have contact Ricci potential, so 
the Theorem  \ref{th-cont-ric-pot} does not apply here. Such manifolds have to be treated separately. The local classification
of manifolds with contact Ricci potential and generalized $\eta$-Einstein comes directly from the Theorem \ref{th-cont-ric-pot}.
\begin{corollary}
 In the denotations of the Theorem \ref{th-cont-ric-pot} for the manifold $\mathcal{M}$ to be generalized $\eta$-Einstein 
it is sufficient and enough that the functions $(f_1,u)$, $(f_2,v)$ be such that 
\begin{equation}
 \begin{array}{l}
f_1 = A_1+By^1, \quad u = C_1-B\dfrac{(y^1)^2}{2}, \\[+8pt]
f_2 = A_2-By^2, \quad v = C_2-B\dfrac{(y^2)^2}{2},
 ×
\end{array}
\end{equation}
 $(A_1,C_1)$, $(A_2,C_2)$ are arbitrary functions of $(x^1,x^2)$ only and $B=B(x^1,x^2)\neq 0$ is nonvanishing everywhere.  Locally
\begin{equation}
 \begin{array}{rcl}
Ric(X,Y) &=& \dfrac{1}{B}\left(g(X,Y)-\eta(X)\eta(Y)\right)+Ric_0, \\[+10pt]
Ric_0 &=& -\dfrac{2}{B}\left(f_{1x^2}+f_{2x^1}\right)\;\theta^1\odot\theta^2,\\
r &=& \dfrac{4}{B}.
 ×
\end{array}
\end{equation}
\end{corollary}

\section{Locally flat manifolds}
\label{secFlat}
\subsection{Local classification}
According to (\ref{curvf}) a manifold is locally flat if and only if the  forms are closed  
\[
 d\tau_1 = 0, \quad d\tau_2 = 0,
\]
and the 2-form $C'=D\omega +\sigma \theta^1\wedge\theta^2$ vanishes $C'=0$. Note that under the assumption 
that $\tau_1$, $\tau_2$ are closed the form $D\omega$ is gauge-invariant. Thus we have to resolve the following 
system of equations 
\begin{equation}
 \label{locfC}
d\tau_1 = 0, \quad d\tau_2 = 0, \quad  D\omega = -\sigma \theta^1\wedge\theta^2.
\end{equation}
By the definition $D\omega = d\omega-\omega\wedge(\tau_1+\tau_2)$, so we have the equation
\[
 d\omega - \omega\wedge(\tau_1+\tau_2) = -\sigma\theta^1\wedge\theta^2,
\]
which yields $d\omega\wedge (\tau_1+\tau_2)=0$ as the right hand of the above equation is a closed form. We have now two possibilities depending on $\tau_1\wedge\tau_2$ is zero or not. So let assume that $\tau_1\wedge\tau_2\neq 0$ on an open
submanifold $\mathcal{U}\subset \mathcal{M}$ then
\[
\tau_1 = \alpha_1\theta^1,\quad \tau_2=\alpha_2\theta^2,\quad \alpha_1\alpha_2\neq 0,
\]
  and $d\omega = f\theta^1\wedge\theta^2$ for a function $f$. Thus
\begin{equation}
\label{Domlflat}
 D\omega = f\theta^1\wedge\theta^2-\omega\wedge(\alpha_1\theta^1+\alpha_2\theta^2) = -\sigma\theta^1\wedge\theta^2,
\end{equation}
and in consequence 
\begin{equation}
\label{exomeg1}
\omega = a\theta^1+b\theta^2+df',  
\end{equation}
 for functions $a,b$, $f'$ and $da = f_1\theta^1\wedge\theta^2$,
$db = f_2\theta^1\wedge\theta^2$, $f = f_2-f_1$. At first we consider the case $df'$ is zero identically, $df'=0$. Under this assumption 
 plugging (\ref{exomeg1}) into (\ref{cartans}) we find that 
\[
 \theta^1\wedge d\theta^3 = 0, \quad \theta^2\wedge d\theta^4 = 0,
\]
hence 
\[
 \theta^3 = u\theta^1+dy^1, \quad \theta^4 = v\theta^2+dy^2,
\]
where $y^1$, $y^2$ are functions. As $(\theta^0=\eta,\theta^1,\ldots,\theta^4)$ is a local 
coframe we find that $\theta^0\wedge\theta^1\wedge\theta^2\wedge dy^1\wedge dy^2 \neq 0$ everywhere, so 
we may pick $(t,x^1,x^2,y^1,y^2)$ as a local coordinates system where $\eta =dt$, $\theta^1=dx^1$, 
$\theta^2 = dx^2$. Now the Cartan equations (\ref{cartans}) impose constraints on functions $(u,v)$, $(\alpha_1,\alpha_2)$
and $(a,b)$ and their derivatives
\begin{equation}
 \begin{array}{c}
du\wedge dx^1 = dx^1\wedge dt +a dx^1\wedge dx^2 -\alpha_1 dx^1\wedge dy^1, \\[+4pt]
dv\wedge dx^2 = \sigma dx^2 \wedge dt +b dx^1 \wedge dx^2 -\alpha_2 dx^2 \wedge dy^2,
 ×
\end{array}
\end{equation}
 Providing similar computations as in the cases of $\eta$-Einstein manifolds or manifolds 
with contact Ricci potential we obtain following  forms for these functions in terms of the given local coordinates 
\begin{equation}
\label{solvlflat}
 \begin{array}{c}
\alpha_1 = \alpha_1(x^1) , \quad \alpha_2 = \alpha_2(x^2), \quad \alpha_1\alpha_2 \neq 0, \\[+4pt]
 u(t,x^1,x^2,y^1) = \alpha_1 y^1 - t +A(x^1,x^2),\quad a(x^1,x^2) = -\partial_{x^2}A, \\[+4pt]
 v(t,x^1,x^2,y^2) = \alpha_2 y^2 -\sigma t +B(x^1,x^2), \quad b(x^1,x^2) = \partial_{x^1}B, 
 ×
\end{array}
\end{equation}
for $D\omega$ we find
\[
   D\omega = \left(\partial^2_{x^2}A+\partial^2_{x^1}B +\alpha_2\partial_{x^2}A+\alpha_1\partial_{x^1}B\right)dx^1\wedge dx^2,
\]
hence according to (\ref{Domlflat}) the functions $ A, B$ have to satisfy the following 
second order linear differential non-homogeneous equation
\begin{equation}
 \label{lflatC1}
  \partial^2_{x^2}A+\partial^2_{x^1}B +\alpha_2\partial_{x^2}A+\alpha_1\partial_{x^1}B = -\sigma,  \quad \sigma = \pm 1.
\end{equation}
Now we are ready to consider general case $\omega=a\theta^1+b\theta^2+df'$ with $df'\neq 0$. We simple notice that it is always possible 
to choose a gauge - to pass to another adopted local coframe according to (\ref{gtran}) - then $\omega$ is changing to a form 
$\omega'$ (\ref{gauge}) and we can assure that $\omega' = a'\theta^1+b'\theta^2$.   

If $\tau_1\wedge \tau_2 = 0$ on an open submanifold $\mathcal{U}\subset \mathcal{M}$ then $\tau_1 = 0$ or $\tau_2=0$ 
on $\mathcal{U}$.  Let assume that both $\tau_1$, $\tau_2$ are vanishing. Then $d\omega = -\sigma \theta^1\wedge\theta^2$ thus 
\begin{equation}
 \label{exomeg2}
\omega = a\theta^1+  b \theta^2 + df',
\end{equation}
 $f'$ is a function on $\mathcal{U}$.  Considering (\ref{solvlflat}) we observe that setting $\alpha_1=0$, $\alpha_2=0$ in (\ref{solvlflat}) we obtain particular  solution with $\tau_1=0$, $\tau_2=0$ and $f'=0$. Now we want to extend this solution for arbitrary $f'$. Our first attempt is to make an ansatz
\begin{equation}
\label{ext3t4}
\theta^3 = (-t+A)\theta^1+f'\theta^2 +dy^1, \quad \theta^4 = (-\sigma t+B)\theta^2-f'\theta^1+dy^2, 
\end{equation}
as before $A=A(x^1,x^2)$, $B=B(x^1,x^2)$, we directly verify that now $(\theta^0=\eta=dt, \theta^1=dx^1,\theta^2=dx^2,\theta^3,\theta^4)$ satisfies the equations 
(\ref{cartans}) with $\omega$ as given in (\ref{exomeg2}) with $a= -\partial_{x^2}A$, $b=\partial_{x^1}B$ and the curvature vanishes if and only if 
\[
 \partial^2_{x^2}A+\partial^2_{x^1}B = -\sigma.
\]
If $\theta^{'3}$, $\theta^{'4}$ are  another local solutions of (\ref{cartans}) with the the same $\omega$ then
differences $\theta^3-\theta^{'3}$, $\theta^4-\theta^{'4}$ are closed forms thus if necessary shrinking $\mathcal{U}$ 
we can write 
\[
 \theta^{'3} = (-t+A)\theta^1+f'\theta^2+dy^1 + dc_1, \quad \theta^{'4} = (-\sigma t+B)\theta^2-f'\theta^1+dy^2 +dc_2,
\]
for functions $c_1$, $c_2$. The coordinate change $ y^1 \mapsto y^{'1} = y^1+c_1$, $y^2 \mapsto y^{'2} = y^2+c_2$ brings the forms $\theta^{'3}$, $\theta^{'4}$ to the same expression as in (\ref{ext3t4}).

Finally let $\tau_1\neq 0$ and $\tau_2=0$ everywhere on $\mathcal{U}$. In this case we find that
\[
 \omega = a\theta^1+df',
\]
for functions $a$, $f'$. Instead of proceeding directly we are looking for a suitable adopted coframe. 
The gauge formula (\ref{gauge}) for $\omega$ simplifies to 
\[
 \omega' = \omega +\alpha \tau_1+d\alpha,
\]
for $\alpha = -a/\alpha_1$ we obtain $\omega' = df'+d\alpha$ and in consequence $d\omega'=0$. As $d\tau_1=d\tau_2=0$
we have $D\omega' = D\omega$ (see Corollary \ref{pKDom}), in such case we say that $D\omega$ is gauge-invariant. 
Summarizing our considerations we see that always there is a local adopted coframe $(\theta^0=\eta, \theta^1,\ldots,
\theta^4)$ with $\omega$ closed. So let fix such adopted coframe then 
\[
 D\omega = -\omega\wedge \tau_1 = -\sigma \theta^1\wedge\theta^2,
\]
 which yields $\omega = a\theta^1-(\sigma/\alpha_1)\theta^2$, for a function $a$, ($\tau_1=\alpha_1\theta^1$),
from $d\omega =0$ we find
\begin{equation}
\label{exa}
 a(x^1,x^2) =-\sigma x^2\partial_{x^1}(\dfrac{1}{\alpha_1}) + C,
\end{equation}
here $C$ is a function of $x^1$ only and as earlier $\theta^1=dx^1$, $\theta=dx^2$. Again the shape of $\omega$ and 
the equations (\ref{cartans}) follow $d\theta^3\wedge\theta^1=0$, $d\theta^4\wedge\theta^2=0$, thus we are looking 
for solutions of the form
\[
 \theta^3 = u\theta^1+dy^1, \quad \theta^4 = v\theta^2+dy^2,
\]
$u$, $v$ are functions determined by (\ref{cartans}) and (\ref{exa}). Some simple considerations follow 
\begin{equation}
 \begin{array}{c}
u= \alpha_1y^1-t +A(x^2)^2+C x^2+D, \quad v= -\sigma(t+s) +B, \\[+6pt]
 \partial_{x^1}s = \dfrac{1}{\alpha_1},\quad  A =\dfrac{\sigma}{2×}\,\partial^2_{x^1}s,\quad s=s(x^1),
 ×
\end{array}
\end{equation}
and $D=D(x^1)$, $B=B(x^2)$.
\begin{theorem}
\label{exthlocd}
 Let $(\mathcal{M},\varphi, \xi, \eta,g)$ be a locally flat manifold. Then near each point $p\in \mathcal{M}$ there are 
local coordinates $(t,x^1,x^2,y^1,y^2)$ and local adopted coframe $(\theta^0=\eta, \theta^1,\ldots,\theta^4)$, such that
\[
\begin{array}{c}
\theta^0 =dt,\quad \theta^1=dx^1,\quad \theta^2 = dx^2, \\[+4pt]
\theta^3= u\theta^1+dy^1, \quad \theta^4= v\theta^2+dy^2,
 ×
\end{array}
\]
the functions $(u,v)$ have general form 
\begin{equation}
\label{exlocuv}
 u = \alpha_1y^1-t +A(x^1,x^2),\quad v = \alpha_2y^2-\sigma t +B(x^1,x^2),
\end{equation}
where
\[
 \tau_1 = \alpha_1\theta^1,\quad \tau_2 = \alpha_2\theta^2, \quad d\tau_1  = d\tau_2 = 0,
\]
and the functions $(A,B)$ satisfy the following condition
\[
 \partial^2_{x^2}A+\partial^2_{x^1}B+\alpha_2\partial_{x^2}A+\alpha_1\partial_{x^1}B = -\sigma.
\]
\end{theorem}

\subsection{Isotropy groups}
By an structure automorphism  (in general local) it is understood an isometry $f$, such that 
\[
 f_*\varphi = \varphi f_*, \quad f_*\xi = \xi,\quad f^*\eta = \eta.
\]
 Respectively an infinitesimal automorphism is a Killing vector field (local)
 $K$ with the properties
\begin{equation}
\label{exinfI}
 \mathcal{L}_K\varphi =0, \quad \mathcal{L}_K \xi = 0,\quad \mathcal{L}_K\eta = 0,
\end{equation}
here $\mathcal{L}_K$ denotes the Lie derivative. The conditions $\mathcal{L}_K\xi=0$ and 
$\mathcal{L}_K\eta=0$ are equivalent. 

Let $\mathcal{M}$ be a flat manifold, for simplicity we assume $\mathcal{M}=\mathcal{U}$ where $\mathcal{U}\subset \mathbb{R}^5$ 
is a domain and we fix a point $O\in \mathcal{U}$ as origin. We also assume that global Cartesian coordinates on $\mathbb{R}^5$ are 
normal flat coordinates: the metric coefficients are constants in these coordinates.

In what will follow we study automorphism with the origin $O$ as a fixed point and infinitesimal automorphisms vanishing at the origin. 
The correspondence $K \mapsto (\nabla K)_O$ is one-one, as we consider only those infinitesimal automorphisms which 
vanishes at $O$, thus $K\mapsto (\nabla K)_O$ defines a representation of a local Lie algebra $\mathfrak{s}$  of all infinitesimal 
automorphisms vanishing at $O$ as a Lie algebra of endomorphisms of the tangent space $T_O\mathcal{U}$. 

Let $f$ be local automorphism defined near the origin, $f:O\mapsto O$.  
\begin{proposition}
Let $(V_0=\xi, V_1, \ldots, V_4)$ be an adopted frame defined near $O \in \mathcal{U}$, 
then 
\begin{equation}
\label{exautf}
 \begin{array}{c}
f_* V_1 = \varepsilon_1 V_1, \quad f_*V_2 = \varepsilon_2 V_2, \\[+4pt]
f_*V_3 = c\varepsilon_2 V_2+\varepsilon_1 V_3, \quad f_*V_4 = -c\varepsilon_1 V_1+\varepsilon_2 V_4, \\[+4pt]
|\varepsilon_1|=|\varepsilon_2|=1.
 ×
\end{array}
\end{equation}
If $K$ is an infinitesimal automorphism, $K=0$ at the origin, then
\begin{equation}
 \label{exinfK}
\begin{array}{c}
\mathcal{L}_K V_1 =0, \quad \mathcal{L}_K V_2 =0, \\[+4pt]
\mathcal{L}_K V_3 = aV_2, \quad \mathcal{L}_K V_4 = -a V_1.
 ×
\end{array}
\end{equation}
\end{proposition}
here $a$, $c$ are some functions.
\begin{proof}
 As $f$ is an local automorphism $f_*\varphi = \varphi f_*$, $f_*\xi = \xi$, thus $\pm 1$-eigenspaces of $\varphi$ are
invariant  $\pm f_*V_i =\varphi f_*V_i$, $i=1,2$. Thus $f_* V_1 = a_1V_1+bV_4$ for functions $a_1$, $b$. We shall show 
that $b=0$ everywhere. Projecting $\nabla_Xf_*V_1$ onto $\xi$ we find 
\[
 g(\xi, \nabla_Xf_*V_1) = g(\xi, \nabla_{f^{-1}_*X}V_1) = \tau_1(f^{-1}_*X)g(\xi,V_1) = 0,
\]
from the other hand 
\[
 \nabla_Xf_*V_1 = (da_1(X)+a_1\tau_1(X))V_1 +db(X)V_4+b\nabla_XV_4,
\]
and according to (\ref{lcf}), $g(\xi,\nabla_Xf_*V_1) = \sigma b g(X,V_2)$, so $b=0$ identically. Exactly in the same way
we prove that $f_*V_2 = a_2 V_2$. As $f$ is an isometry
\[
 \begin{array}{c}
f_*V_1 = a_1V_1, \quad f_*V_2 = a_2V_2, \\[+4pt]
f_*V_3 = a_2cV_2+\dfrac{1}{a_1}V_3, \quad f_* V_4 = -a_1c V_1+\dfrac{1}{a_2}V_4,
 ×
\end{array}
\]
$c$ denotes a functions, note that
\[
 g(\xi, \nabla_Xf_*V_3) = g(\xi, \nabla_{f^{-1}_*X}V_3) = g(f^{-1}_*X,V_1) = a_1g(X,V_1),
\]
and from the other hand as $\nabla_X f_*V_3 = \ldots + \dfrac{1}{a_1}\nabla_XV_3$
\[
 g(\xi, \nabla_Xf_*V_3) = \dfrac{1}{a_1}g(\xi,\nabla_XV_3) = \dfrac{1}{a_1}g(X,V_1),
\]
in consequence $a_1 = \dfrac{1}{a_1}$ and $a_1 =\pm 1$, similarly we obtain $a_2=\pm 1$. This proves (\ref{exautf}). 
Let $f^t$ be a local   1-parameter group of automorphisms, $f^t:O\mapsto O$, as $f^0_*=Id$ is the identity transformation 
applying (\ref{exautf}) to each $f^t_*$ we must have $\varepsilon_1=\varepsilon_2=1$ and
\[
f^t_*V_1 = V_1,\quad f^t_*V_2= V_2, \quad f^t_*V_3 =c(t)V_2+V_3, \quad f^t_*V_4 = -c(t)V_1+V_4,
\]
if $K$ is a respective infinitesimal automorphism, $K(x) = \dfrac{d}{dt}f^tx|_{t=0}$, then for example
\[
 \mathcal{L}_K V_1 = \lim\limits_{t\rightarrow 0}\dfrac{1}{t}[V_1-f^{t}_*V_1] =0,
\]
and similarly we find $\mathcal{L}_KV_2=0$, $\mathcal{L}_KV_3=aV_2$, $\mathcal{L}_KV_4=-aV_1$.   
\end{proof}
\begin{proposition}
\label{exiso}
 If $K$ is an infinitesimal automorphism vanishing the origin, $K(O)=0$, $\mathcal{L}_K\theta^3 = a\theta^2$ then in coordinates as 
in the Theorem \ref{exthlocd}
\begin{equation}
\begin{array}{c}
K = k_1(x^1,x^2)\partial_{y^1}+k_2(x^1,x^2)\partial_{y^2},\\[+4pt]
 dk_1 = -\alpha_1 k_1 dx^1 +adx^2, \quad dk_2 = -\alpha_2 k_2dx^2 -adx^1, \\[+4pt]
\partial_{x^1}a = -\alpha_1 a,\quad \partial_{x^2}a = -\alpha_2 a,
 ×
\end{array}
\end{equation}
The local Lie algebra of all infinitesimal automorphisms vanishing at $O$ is 1-dimensional, equivalently for 
the infinitesimal automorphisms $K_1$, $K_2$ vanishing at $O$ there is a constant $c_{12}$ such that
$K_1 = c_{12}K_2$.
\end{proposition}
\begin{proof}
 As $K$ is Killing vector field the equations (\ref{exinfK}) in metric-dual coframe $(\theta^0=\eta,\theta^1,\ldots,\theta^4)$ reads as
\[
 \begin{array}{c}
\mathcal{L}_K\theta^0 = 0,\quad \mathcal{L}_K\theta^1 = 0,\quad \mathcal{L}_K\theta^2 = 0, \\[+4pt]
\mathcal{L}_K\theta^3 = a\theta^2, \quad \mathcal{L}_K\theta^4 = -a\theta^1,
 ×
\end{array}
\]
as the 1-forms $(\theta^0,\theta^1,\theta^2)$ are closed the functions $\theta^0(K)$, $\theta^1(K)$, 
$\theta^2(K)$ are constants and thus they identically zero for they vanish at the origin $O$.
In the local coordinates $(t,x^1,x^2,y^1,y^2)$ from the Theorem \ref{exthlocd} 
\[
 K = k_1\partial_{y^1}+k_2\partial_{y^2},
\]
and the equation $\mathcal{L}_K\theta^3 = a\theta^2$ yields 
\[
\mathcal{L}_K\theta^3=\mathcal{L}_K(u\theta^1+dy^1) = (\mathcal{L}_Ku)\theta^1+dk_1 = adx^2,    
\]
thus $dk_1 = -du(K)dx^1+adx^2$ which proves $k_1 = k_1(x^1,x^2)$ and analogously for $k_2$, 
$dk_2=-dv(K)dx^2-adx^1$, $k_2 = k_2(x^1,x^2)$.  According to (\ref{exlocuv}) 
\begin{equation}
\label{exdk12}
 dk_1 = -\alpha_1 k_1 dx^1 +adx^2, \quad dk_2 = -\alpha_2 k_2dx^2 -adx^1,
\end{equation}
and in consequence for the function $a$ we find
\begin{equation}
\label{exapde}
 \partial_{x^1}a = -\alpha_1 a,\quad \partial_{x^2}a = -\alpha_2 a,
\end{equation}
the last identity follows that given two infinitesimal automorphisms $K_1$, $K_2$ and functions $a_1$, $a_2$,
$\mathcal{L}_{K_1}\theta^3 = a_1\theta^2$, $\mathcal{L}_{K_2}\theta^3 = a_2\theta^2$, the proportion $a_1/a_2$ is 
a constant, $a_1/a_2 = const.$ Explicitly the solution of (\ref{exapde}) can be given  the form
\[
\mathcal{U}\ni p \mapsto a(p) = Ae^{-\int\limits_\gamma \tau_1+\tau_2}, \quad A = a(O)= const,
\]
where $\gamma\subset \mathcal{U}$ is an arbitrary (smooth) curve joining the origin and the point $p$. If the function $a$ vanishes at the origin it vanishes everywhere and this is possible if and only if the corresponding infinitesimal automorphism is null, i.e. $K=0$ identically. Indeed, we have $\mathcal{L}_KV_i = 0$, $i=0,\ldots,4$,
 and the covariant derivative $\nabla K$ vanishes at the origin but then as both $K$ and $\nabla K$ vanish at the origin, 
$K$ vanishes near $O$ and thus everywhere as $\mathcal{U}$ is connected. We fix now the function $a$ requiring 
$a(O)=1$ and note this as $a_0$ thus (\ref{exdk12}) takes the form
\[
 dk_1 = -\alpha_1k_1 dx^1 +Ca_0dx^2, \quad dk_2 = -\alpha_2 k_2 dx^2 -Ca_0dx^1, \quad C=const \neq 0,
\]
 a solution with $C=1$ we call the normalized solution. If $(k_1^1,k_2^1)$ and $(k_1^2,k_2^2)$ are normalized solutions 
then the difference $(k_1^1-k_1^2,k_2^1-k_2^2)$ is a solution of the same equations as (\ref{exapde}) vanishing at $O$ which follows that   $(k_1^1,k_2^1)$ and $(k_2^1,k_2^2)$ coincide everywhere.  
\end{proof}

\subsection{Manifolds with transitive automorphism groups}
Here we provide a detailed discusion of the example of a flat manifold  from \cite{Dtsu}. At fist we recall the definition

\begin{example}
\label{exex}
 Let $\mathcal{M}=\mathbb{R}^5$, $\mathcal{M}\ni p = (z,u_1,u_2,v_1,v_2)$, an almost para-contact structure 
$(\varphi, \xi, \eta, g)$ is defined on $\mathcal{M}$ by
\begin{equation}
\label{exE1}
 \begin{array}{c}
\varphi \dfrac{\partial}{\partial z} = 2u_1\dfrac{\partial}{\partial v^1} -2u_2\dfrac{\partial}{\partial v_2},\\[+8pt]
\varphi \dfrac{\partial}{\partial v_1} = \dfrac{\partial}{\partial v_1}, \quad 
\varphi \dfrac{\partial}{\partial v_2} = - \dfrac{\partial}{\partial v_2}, \\[+8pt]
\varphi \dfrac{\partial}{\partial u_1} = -2u_1\dfrac{\partial}{\partial z} -\dfrac{\partial}{\partial u_1} 
+ 4u_1u_2\dfrac{\partial}{\partial v_2}, \\[+8pt]
\varphi \dfrac{\partial}{\partial u_2} = 2u_2\dfrac{\partial}{\partial z} +\dfrac{\partial}{\partial u_2} 
-4u_1u_2 \dfrac{\partial}{\partial v_1}, \\[+8pt]
\xi = \dfrac{\partial}{\partial z} - 2u_1\dfrac{\partial}{\partial v_1} - 2u_2\dfrac{\partial}{\partial v_2},\\[+8pt]
\eta = dz - 2u_1du_1 - 2u_2du_2, \\[+8pt]
g = dz^2 +2du_1dv_1+2du_2dv_2,
 ×
\end{array} 
\end{equation}
we directly verify that $\mathcal{M}$ equipped with this structure becomes a weakly para-cosymplectic manifold with
para-K\"ahler leaves, evidently $\mathcal{M}$ is flat as pseudo-Riemannian manifold, for the tensor $A=-\nabla\xi$, we have
\[
 A = 2du_1\otimes\dfrac{\partial}{\partial v_1}+2du_2\otimes\dfrac{\partial}{\partial v_2},
\]
so $\mathcal{M}$ is non-paracosymplectic and of elliptic type. Introducing the adopted
frame $(V_0=\xi)$
\[
 \begin{array}{c}
V_1 = \sqrt{2}\dfrac{\partial}{\partial v_2}, \quad V_2 = \sqrt{2}\dfrac{\partial}{\partial v_1}, \\[+8pt]
V_3 = \sqrt{2}u_2\dfrac{\partial}{\partial z} +\dfrac{1}{\sqrt{2}}\dfrac{\partial}{\partial u_2} - \sqrt{2}(u_2)^2\dfrac{\partial}{\partial v_2}, \\[+8pt]
V_4 = \sqrt{2}u_1\dfrac{\partial}{\partial z} +\dfrac{1}{\sqrt{2}}\dfrac{\partial}{\partial u_1} 
-\sqrt{2}(u_1)^2\dfrac{\partial}{\partial v_1},
 ×
\end{array}
\]
we find that the only non-vanishing commutators from $[V_i,V_j]$, $i,j=0,\ldots,4$ are
\[
  \quad [\xi, V_3] = V_1, \quad [\xi, V_4] = V_2,
\]
thus the frame constitutes a base for a 5-dimensional Lie algebra of vector fields. Let $\mathcal{G}$ be a Lie group 
with the Lie algebra $\mathfrak{g}$ isomorphic to the one spanned by $(\xi,V_1,\ldots,V_4)$ thus we can export the 
structure $(\varphi,\xi,\eta,g)$ to a left-invariant structure $(\tilde\varphi, \tilde\xi,\tilde\eta,\tilde g)$ defined on $\mathcal{G}$ and these structures are locally isomorphic. In particular $\tilde g$ is flat left-invariant pseudo-metric on the non-abelian Lie group $\mathcal{G}$. We note that $\mathcal{G}$ acts on itself by left-translations as group of automorphisms of the structure. So, $\mathcal{G}$ is an example of flat weakly 
para-cosymplectic manifold, of elliptic type,  with a transitive group of automorphisms.    
\end{example}

 First, note that we may identify the manifold $\mathcal{M}$ which appears above with an open neighborhood $\mathcal{U}_e = \mathcal{M}\subset \mathcal{G}$ of the identity element 
$e\in \mathcal{G}$, $(z,u_1,u_2,v_1,v_2)$ are then local coordinates near $e=(0,0,0,0,0)$ and (\ref{exE1}) 
is a local description of the left-invariant structure $(\tilde\varphi,\tilde\xi,\tilde\eta,\tilde g)$ 
on $\mathcal{U}$. Let introduce new coordinates $(x^1,\ldots,x^5)$ on $\mathcal{M}$
\[
 x^1=z,\quad x^2 =\dfrac{1}{\sqrt{2}}v_2, \quad x^3=\dfrac{1}{\sqrt{2}}v_1, \quad x^4=\sqrt{2}u_2,\quad x^5 = \sqrt{2}u_1,
\]
then $g=(dx^1)^2 + 2dx^2dx^4+2dx^3dx^5$, now we consider a group of affine maps of the form 
\[
 (\mathbb{A}, \vec{p})=((a^i_j), (p^i)): x^i \mapsto x^{'i} = a^i_jx^i+p^j, \quad i,j=1,\ldots,5,
\]
assuming $(\mathbb{A},\vec{p})$ is  an automorphism of the structure $(\varphi,\xi,\eta,g)$. Explicitly for maps from the 
connected component of the identity we find
\[
 \begin{array}{l}
x^1 \mapsto x^{'1} = x^1+p^4x^4+p^5x^5+p^1,\\[+4pt]
x^2 \mapsto x^{'2} = -p^4x^1+x^2-(p^4)^2x^4/2+a^2_5x^5+p^2, \\[+4pt]
x^3 \mapsto x^{'3} = -p^5x^1+x^3 -(p^4p^5+a^2_5)x^4-(p^5)^2x^5/2+p^3, \\[+4pt]
x^4 \mapsto x^{'4} = x^4+p^4,\\[+4pt]
x^5 \mapsto x^{'5} = x^5+p^5,
 ×
\end{array}
\]
here $a^2_5$ simply denotes arbitrary real parameter and $(p^1,\ldots,p^5)$ are components of the translation vector $\vec{p}$.  
Using a representation as a $6\times 6$-matrices 
\[
  A = 
\begin{pmatrix}
 \mathbb{A} & \vec{p} \\
  0         &  1
\end{pmatrix}
\]
the generators of the respective Lie algebra are values of the canonical right-invariant Cartan-Maurer form $\omega = dA\cdot A^{-1}$, 
thus we have 
\[
 K_i = \omega(\partial_{p^i})|_{e}= \partial_{p^i}A\cdot A^{-1}|_{e}, \; i=1,\ldots,5, 
\quad K_6 = \omega(\partial_{a^2_5})|_{e} = \partial_{a^2_5}A\cdot A^{-1}|_{e},
\]
$e=(0,\ldots,0)$, by $\exp(tK_i)$ we denote the respective 1-parameter groups. Each $\exp(tK_i)$ we treat as a 1-parameter
group of affine maps of $\mathcal{M}$ and these affine maps are automorphisms of the structure, for example $\exp(tK_4)$ 
are of the form
\[
\begin{array}{l}
x^1 \mapsto x^{'1}= x^1+tx^4+t^2/2, \\[+4pt]
 x^2 \mapsto x^{'2} = -tx^1 + x^2 -t^2x^4/2 - t^3/6, \\[+4pt]
 x^3 \mapsto x^{'3} = x^3, \\[+4pt]
 x^4 \mapsto x^{'4} = x^4+t, \\[+4pt]
 x^5 \mapsto x^{'5} = x^5,
 ×
\end{array}
\]
if $\tilde K_i$ are the infinitesimal automorphisms defined by these 1-parameter groups then we find 
\[
 \begin{array}{c}
\tilde K_1 = \partial_{x^1},\quad \tilde K_2 = \partial_{x^2},\quad \tilde K_3 = \partial_{x^3}, \\[+4pt]
\tilde K_4 = x^4\partial_{x^1}-x^1\partial_{x^2}+\partial_{x^4}, \quad 
\tilde K_5 = x^5\partial_{x^1}-x^1\partial_{x^3}+\partial_{x^5}, \\[+4pt]
\tilde K_6 = x^5\partial_{x^2}-x^4\partial_{x^3},
 ×
\end{array}
\]
we now easily obtain the structure for the Lie algebra spanned by $(\tilde K_1,\ldots, \tilde K_6)$, for the non-zero 
commutators are
\[
 \begin{array}{c}
[\tilde K_1, \tilde K_4] = -\tilde K_2, \quad [\tilde K_1, \tilde K_5] = -\tilde K_3, \\[+4pt]
[\tilde K_4, \tilde K_5] = \tilde K_6, \quad [\tilde K_4, \tilde K_6] = -\tilde K_3, \quad 
[\tilde K_5, \tilde K_6] = \tilde K_2 
 ×
\end{array}
\]
we denote this Lie algebra by $\mathfrak{s}$, we note that $\mathfrak{s}$ is a 3-step nilpotent, the vector field 
$\tilde K_6$ it is exactly the generator of the isotropy group described in the Proposition \ref{exiso}.
Summarizing our discussion we notice that the problem of a description of the Lie groups of automorphisms of the manifold
$\mathcal{M}$ described in the Example \ref{exex} is non-trivial despite the fact the metric is flat. We notice that $\mathcal{M}$ 
admits at least two non-isomorphic effective groups of automorphisms: the subgroup of affine isometries described by the Lie 
algebra $\mathfrak{s}$ of infinitesimal automorphisms and a 1-dimensional normal extention of the Lie group $\mathcal{G}$ described 
in the Example \ref{exex}. Also note that as a corollary we obtain that the left action of the group $\mathcal{G}$ on itself is 
non-affine; left translations in general can not be expressed  as affine maps in the given coordinates.
    
\section{ Left-invariant structures}
 Given a left-invariant weakly para-cosymplectic structure $(\phi, \xi , \eta, g)$ on a Lie group $\mathcal{G}$. Then 
an adopted frame consists from left-invariant vector fields therefore  the forms $\tau_1$, $\tau_2$ and 
$\omega$ appearing  in the Proposition \ref{lc}  are all left-invariant.
In consequence functions $\alpha_1$ and $\alpha_2$ are constants.   The Proposition {\bf \ref{lc}} yields  the following system of the structure equations for the Lie algebra of $\mathcal{G}$ 
\begin{equation}
\label{streq}
\begin{array}{l}
 d\eta  =  d\theta^1 = d\theta^2 = 0,  \\[+2pt]
 d\theta^3  =  \theta^1\wedge \eta - \alpha_1\theta^1\wedge \theta^3 + \omega \wedge \theta^2, \\[+2pt]
 d \theta^4  =\sigma\theta^2\wedge \eta - \alpha_2\theta^2\wedge \theta^4 - \omega \wedge \theta^1, 
\end{array}
\end{equation}
or with respect to the dual frame $\gamma^i(V_j)=\delta^i_j$, $i,j=0,\ldots,4$
\begin{equation}
\begin{array}{l}
 d\gamma^0  =  d\gamma^3 = d\gamma^4 = 0,  \\[+2pt]
 d\gamma^1  =  \gamma^3\wedge \gamma^0 - \alpha_1\gamma^3\wedge \gamma^1 + \omega \wedge \gamma^4, \\[+2pt]
 d \gamma^2  =\sigma\gamma^4\wedge \gamma^0 - \alpha_2\gamma^4\wedge \gamma^2 - \omega \wedge \gamma^3.
\end{array}
\end{equation}
 Applying the exterior derivative we get 
the  integrability conditions
\begin{equation}
\label{lhintcond}
\begin{array}{l}
 0 = d\omega \wedge \theta^2 - \alpha_1 \omega \wedge \theta^1\wedge \theta^2, \\[+2pt]
0 = d\omega \wedge \theta^1 + \alpha_2 \omega \wedge \theta^1\wedge \theta^2.
\end{array}
\end{equation}
Let
\begin{equation}
\label{lhomeg} 
\omega = \alpha_0\eta + \beta_1\theta^1+\beta_2\theta^2+\alpha_3 \theta^3 + \alpha_4\theta^4 = 
\alpha_0\gamma^0 + \alpha_3 \gamma^1 + \alpha_4\gamma^2+\beta_1\gamma^3+\beta_2\gamma^4,
\end{equation}
here $\alpha_0$, $\alpha_3$, $\alpha_4$ and $\beta_1$, $\beta_2$ are real constants. Plugging (\ref{lhomeg}) 
 into  (\ref{lhintcond}) we get the following constraints on the constants $\alpha_0$,  $\alpha_1$, $\alpha_2$,  $\alpha_3$, $\alpha_4$
\begin{equation}
 \begin{array}{l}
\alpha_0 (\alpha_1 + \alpha_4) = -\alpha_3, \quad\quad
\alpha_0 (\alpha_2 -\alpha_3) = -\sigma\alpha_4, \\
\alpha_3 (\alpha_2 -\alpha_3) = 0, \quad\quad 
\alpha_4 (\alpha_1 + \alpha_4) = 0, \\
\alpha_3\alpha_4 = 0,
\end{array}
\end{equation}
 we have the following explicit solutions and the corresponding forms ($d\gamma^0=d\gamma^3=d\gamma^4=0$ in the all cases) :
\begin{itemize}
 \item[(A)] $\alpha_3=0$, $\alpha_4\neq 0$,
\begin{equation*}
\begin{array}{c}
 \omega = \sigma\dfrac{\alpha_1}{\alpha_2}\gamma^0 -\alpha_1 \gamma^2 +\beta_1\gamma^3+\beta_2\gamma^4, \quad \tau_1\wedge\tau_2 \neq 0,\\[+4pt]
  d\gamma^1  =  -\gamma^0\wedge \gamma^3 +\sigma\dfrac{\alpha_1}{\alpha_2}\gamma^0\wedge\gamma^4 +\alpha_1\gamma^1\wedge\gamma^3
  -\alpha_1\gamma^2\wedge\gamma^4+\beta_1\gamma^3\wedge\gamma^4,\\[+4pt]
d \gamma^2  = -\sigma \dfrac{\alpha_1}{\alpha_2}\gamma^0\wedge\gamma^3 -\sigma\gamma^0\wedge\gamma^4 +\alpha_1\gamma^2\wedge\gamma^3 
+\alpha_2\gamma^2\wedge\gamma^4+\beta_2\gamma^3\wedge\gamma^4,
\end{array}
\end{equation*}

\item[(B)] $\alpha_3 \neq 0$, $\alpha_4 = 0$,
\begin{equation*}
\begin{array}{c}
 \omega = -\dfrac{\alpha_2}{\alpha_1}\gamma^0 +\alpha_2 \gamma^1+\beta_1\gamma^3+\beta_2\gamma^4, \quad \tau_1\wedge\tau_2 \neq 0,\\[+4pt]
   d\gamma^1  = -\gamma^0\wedge\gamma^3 -\dfrac{\alpha_2}{\alpha_1}\gamma^0\wedge\gamma^4+\alpha_1\gamma^1\wedge\gamma^3
  + \alpha_2\gamma^1\wedge\gamma^4+\beta_1\gamma^3\wedge\gamma^4,\\[+4pt] 
 d \gamma^2  = \dfrac{\alpha_2}{\alpha_1}\gamma^0\wedge\gamma^3 -\sigma \gamma^0\wedge\gamma^4 -\alpha_2\gamma^1\wedge\gamma^3
+\alpha_2\gamma^2\wedge\gamma^4+\beta_2\gamma^3\wedge\gamma^4,
\end{array}
\end{equation*}

\item[(C1)] $\alpha_3=\alpha_4=0$,
\begin{equation*}
\begin{array}{c}
\omega = \alpha_0 \gamma^0+\beta_1\gamma^3+\beta_2\gamma^4, \quad \alpha_0 \neq 0, \quad \tau_1=\tau_2=0,   \\[+4pt]
d\gamma^1 = -\gamma^0\wedge\gamma^3+\alpha_0\gamma^0\wedge\gamma^4+ \beta_1\gamma^3\wedge\gamma^4, \\[+4pt]
d\gamma^2 = -\alpha_0\gamma^0\wedge\gamma^3-\sigma \gamma^0\wedge\gamma^4+\beta_2\gamma^3\wedge\gamma^4, 
\end{array}
\end{equation*}
\item[(C2)] $\alpha_0=\alpha_3=\alpha_4=0$ and 
\begin{equation*}
\begin{array}{c}
 \omega =  \beta_1\gamma^3+\beta_2\gamma^4, \quad \tau_1,\;\tau_2 \textrm{ - are arbitrary}, \\[+4pt]
d\gamma^1 = -\gamma^0\wedge\gamma^3+\alpha_1\gamma^1\wedge\gamma^3+\beta_1\gamma^3\wedge\gamma^4, \\[+4pt]
d\gamma^2 = -\sigma \gamma^0\wedge\gamma^4+\alpha_2\gamma^2\wedge\gamma^4+\beta_2\gamma^3\wedge\gamma^4.
\end{array}
\end{equation*}
\end{itemize}

According to these cases (A), (B), (C1) and (C2) we introduce four families of  Lie groups $\mathcal{G}_1$, $\mathcal{G}_2$, $\mathcal{G}_3$, $\mathcal{G}_4$ with the respective Lie algebras: 
\begin{equation*}
 \begin{array}{l}
\textrm{(A)}\;\mathfrak{g}_1:\quad\begin{cases}
  \left[ \xi, V_3 \right] = V_1 +\sigma\dfrac{\alpha_1}{\alpha_2}V_2, \quad
  \left[ V_1, V_3 \right] = -\alpha_1 V_1, \quad 
  \left[ V_2, V_3 \right] = -\alpha_1 V_2, \\[+10pt]
\left[ \xi, V_4 \right] = -\sigma\dfrac{\alpha_1}{\alpha_2} V_1+\sigma V_2, \quad
\left[ V_2, V_4 \right] = \alpha_1V_1-\alpha_2 V_2, \quad [V_3, V_4] = -\beta_1V_1-\beta_2 V_2,  
\end{cases} \\[+30pt]
\textrm{(B)}\;\mathfrak{g}_2:\quad\begin{cases}
  \left[ \xi, V_3\right] = V_1 - \dfrac{\alpha_2}{\alpha_1}V_2, \quad
  \left[ V_1, V_3 \right] = -\alpha_1 V_1 +\alpha_2 V_2, \quad [V_3,V_4] = -\beta_1V_1-\beta_2V_2, \\[+10pt]
\left[ \xi, V_4 \right] = \dfrac{\alpha_2}{\alpha_1} V_1+\sigma V_2, \quad
\left[ V_1, V_4 \right] = -\alpha_2 V_1, \quad
\left[ V_2, V_4 \right] = -\alpha_2 V_2,  
\end{cases} \\[+30pt]
\textrm{(C1)}\;\mathfrak{g}_4:\quad\begin{cases} 
 \left[ \xi, V_3\right] = V_1 +\alpha_0V_2, \quad
  \left[ \xi, V_4 \right] = -\alpha_0 V_1+\sigma V_2, \quad [V_3,V_4] = -\beta_1V_1-\beta_2V_2,
\end{cases}\\[+15pt]
\textrm{(C2)}\;\mathfrak{g}_3:\quad\begin{cases}
  \left[\xi,  V_3\right] = V_1,\quad
  \left[ V_1, V_3 \right] = -\alpha_1 V_1, \\[+10pt]
\left[\xi,  V_4 \right] = \sigma V_2, \quad
\left[ V_2, V_4 \right] = -\alpha_2 V_2,\quad [V_3,V_4] = -\beta_1V_1-\beta_2V_2.  
\end{cases}
\end{array}
\end{equation*}
We summarize our considerations in the following manner
\begin{theorem}
 If there is a left-invariant structure $(\phi, \xi, \eta, g)$, so all the tensors are left-invariant, on a Lie group 
$\mathcal{G}$, making $\mathcal{G}$  a hyperbolic or elliptic  weakly para-cosymplectic manifold 
with para-K\"ahler leaves, then $\mathcal{G}$ belongs to the one of the  classes $\mathcal G_1$, $\mathcal G_2$, $\mathcal G_3$,
$\mathcal G_4$. 
\end{theorem}
\begin{remark}
 The Example from the previous section is a Lie group from the class (C2) with $\alpha_1=\alpha_2=0$, $\beta_1=\beta_2=0$ 
and $\sigma=1$. 
Note also that the groups always appears in the dual pairs where duality is established by the value of $\sigma=\pm 1$.
\end{remark}

\end{document}